\documentclass[review]{elsarticle}

\usepackage{lineno,hyperref}
\modulolinenumbers[5]
\usepackage{amssymb}
\usepackage{amsmath,amssymb,amsfonts,amsthm,graphics,
latexsym, amscd, amsfonts, epsfig, eepic,epic,psfrag,natbib,caption}
\usepackage{mathrsfs}
\usepackage{color}
\usepackage{eucal}%    caligraphic-euler fonts: \mathcal{ }
\usepackage{eufrak}%   frak-euler        fonts: \mathfrak{ }
\usepackage[all]{xypic}
\usepackage{xspace}
\usepackage{algorithm}
\usepackage{algorithmic}
\journal{Mathematical Biosciences}

\theoremstyle{plain}
\newtheorem{theorem}{Theorem}
\newtheorem{corollary}{Corollary}

\newtheorem{lemma}{Lemma}

\theoremstyle{definition}
\newtheorem{definition}{Definition}

\theoremstyle{example}

\theoremstyle{remark}

\numberwithin{equation}{section}

\begin{document}

\begin{frontmatter}

\title{On RNA-RNA interaction structures of fixed
       topological genus}

%\address{}

\author[mymainaddress]{Benjamin M.M.~Fu}
\ead{benjaminfmm@imada.sdu.dk}

\author[mymainaddress]{Hillary S.W.~Han}
\ead{hillary@imada.sdu.dk}

\author[mymainaddress]{Christian M.~Reidys\corref{mycorrespondingauthor}}
\cortext[mycorrespondingauthor]{Corresponding author}
\ead{duck@santafe.edu}

\address[mymainaddress]{Department of Mathematics and Computer science, 
University of Southern Denmark\\
Campusvej 55, DK-5230 Odense M, Denmark}
\address{}

\begin{abstract}
Interacting RNA complexes are studied via bicellular maps
using a filtration via their topological genus.
Our main result is a new bijection for RNA-RNA interaction structures
and linear time uniform sampling algorithm for RNA complexes of fixed
topological genus.
The bijection allows to either reduce the topological genus of a
bicellular map directly, or to lose connectivity by decomposing the 
complex into a pair of single stranded RNA structures.
Our main result is proved bijectively. It provides an explicit algorithm
of how to rewire the corresponding complexes and an unambiguous 
decomposition grammar. 
Using the concept of genus induction, we construct bicellular maps of fixed 
topological genus $g$ uniformly in linear time. 
We present various statistics on these topological RNA complexes and compare
our findings with biological complexes. Furthermore we show how to 
construct loop-energy based complexes using our decomposition grammar.
\end{abstract}


\end{frontmatter}

\linenumbers

\section{Introduction}\label{S:Introduction}

RNA-RNA interactions constitute one of the fundamental mechanisms of
cellular regulation. We find such interactions in a variety of contexts: 
small RNAs binding a larger (m)RNA target including: the regulation 
of translation in both prokaryotes \citep{Vogel:07} and eukaryotes 
\citep{McManus,Banerjee}, the targeting of chemical modifications 
\citep{Bachellerie}, insertion editing \citep{Benne} and transcriptional 
control \citep{Kugel}.

A salient feature is the formation of RNA-RNA interaction structures
that are far more complex than simple sense-antisense interactions.
This is observed for a vast variety of RNA classes including miRNAs,
siRNAs, snRNAs, gRNAs, and snoRNAs.
Thus deeper understanding of RNA-RNA interactions in terms of
the thermodynamics of binding and in its structural consequences is a
necessary prerequisite to understanding RNA-based regulation mechanisms.

An RNA molecule is a linearly oriented sequence of four types of nucleotides,
namely,
{\bf A}, {\bf U}, {\bf C}, and {\bf G}. This sequence is endowed with a
well-defined orientation from the $5'$- to the $3'$-end and referred to as
the backbone.
Each nucleotide can form a base pair by interacting with at most one other
nucleotide by establishing hydrogen bonds. Here we restrict ourselves to
Watson-Crick base pairs {\bf GC} and {\bf AU} as well as the wobble base
pairs {\bf GU}. In the following, base triples as well as other types of more
complex interactions are neglected.

RNA structures can be presented as diagrams by drawing the backbone
horizontally and all base pairs as arcs in the upper half-plane;
see Figure~\ref{F:RNAp}.
This set of arcs provides our coarse-grained RNA structure in
particular ignoring any spatial embedding or geometry of the molecule
beyond its base pairs.
%%%%%%%%%
%%%%%%%%%%%%%%%%%%%%%%%%%%%%%%%%%%%%%%%%%%%%%%%%%%
%%%%%%%%%

%%%%%%%%
 \begin{figure}[ht]
 \begin{center}
 \includegraphics[width=0.9\textwidth]{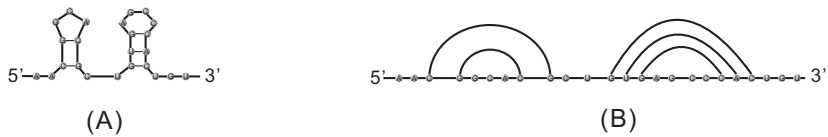}
 \end{center}
 \caption{\small (A) An RNA secondary structure and (B) its diagram
 representation.
 }\label{F:RNAp}
 \end{figure}

%%%%%%%%%
%%%%%%%%%%%%%%%%%%%%%%%%%%%%%%%%%%%%%%%%%%%%%%%%%%
%%%%%%%%%
As a result, specific classes of base pairs translate into distinct
structure categories, the most prominent of which are secondary
structures \citep{Kleitman:70,Nussinov:1978,Waterman:78a,Waterman:79a}.
Represented as diagrams, secondary structures have only non-crossing
base pairs (arcs).
Beyond RNA secondary structures are the RNA pseudoknot structures that
allow for cross serial interactions \citep{Rivas:99}. Once such cross serial
interactions are considered the question of a meaningful filtration arises,
since the folding of unconstrained pseudoknot structures is NP-hard
\citep{Lyngso}. Based on  several earlier studies of the genus of 
a pseudoknot single strand of RNA \citep{vernizzi, vernizzib,bon,andersen2011},
there are several meaningful filtrations of cross-serial
interactions \citep{Orland:02,Reidys:11a,Reidys:10w}.

RNA interaction structures are diagrams over two backbones. Distinguishing
internal and external arcs, the former being arcs within one backbone and
the latter connecting the backbones, interaction structures can be
represented by drawing the two backbones on top of each other, see
Figure~\ref{F:diag_represent}.
%%%%%%%%%
%%%%%%%%%%%%%%%%%%%%%%%%%%%%%%%%%%%%%%%%%%%%%%%%%%
%%%%%%%%%

%%%%%%%%%
 \begin{figure}[ht]
 \begin{center}
 \includegraphics[width=0.9\textwidth]{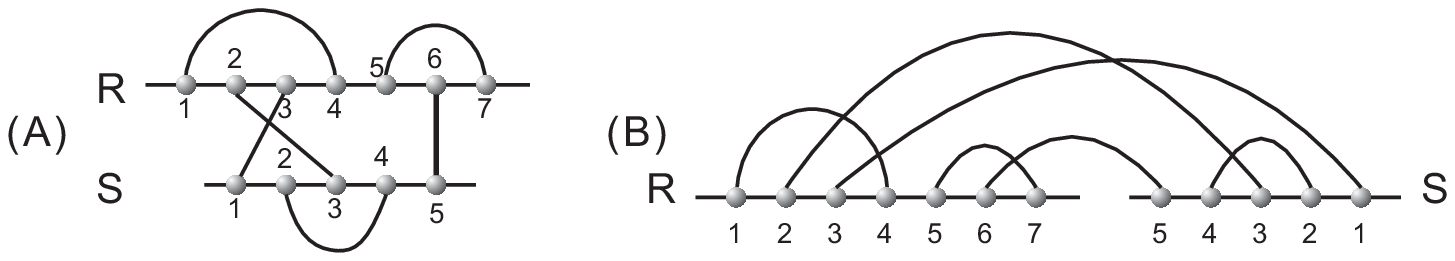}
 \end{center}
 \caption{\small Diagram representation of an RNA-RNA interaction
 structure.
 }\label{F:diag_represent}
 \end{figure}

%%%%%%%%%
%%%%%%%%%%%%%%%%%%%%%%%%%%%%%%%%%%%%%%%%%%%%%%%%%%
%%%%%%%%%

This paper will utilize a topological filtration to categorize 
RNA-complexes. While the basic concept of fat graphs employed here dates 
back to Cayley, the classification and expansion of pseudoknotted RNA 
structures in terms of topological genus of a fat graph or double line graph
were first proposed by \citep{Orland:02} and \citep{Bon:08}.
Fat graphs were applied to RNA secondary structures even earlier in
\citep{Waterman:93} and \citep{Penner:03}.  
The results of \citep{Orland:02} are based on the matrix models and
are conceptually independent. Genus, as well as other 
topological invariants of fat graphs were introduced and studied as 
descriptors of proteins in \citep{protein}.

The approach undertaken here is combinatorial and follows \citep{Fenix:t2}: 
starting with the diagram representation we inflate each edge, including 
backbone edges, into ribbons. As each ribbon has
two sides and specifying a counter-clockwise rotation around each vertex,
we obtain so called boundary cycles with a unique orientation. It is clear
that we have thus constructed a surface and its topological genus
providing the filtration. Naturally there are many such ribbon graphs
that produce the same topological surface (by gluing the two ``complementary''
sides of each ribbon), this is how we obtain the desired equivalence
(complexity) classes of structures.

The idea of genus induction is an extension of the framework of
\citep{chapuy:10,Chapuy:11}, who studied unicellular maps of genus $g$. In
\citep{fenix} a linear time algorithm for uniformly generating RNA structures
of fixed topological genus was presented employing the results of 
\citep{Chapuy:11}. In \citep{fenix-shape} this framework was extended to deal
directly with RNA-shapes, i.e.~enabling the uniform generation of finitely many
shapes for fixed topological genus and to thereby extract key information from
RNA databases.

In this contribution we derive the theory of RNA-RNA interaction structures
by means of a new recursion. In the course of its construction 
we have to deal with the fact that it is not a ``pure''. This means
it involves not only bicellular maps of lower genus but also disjoint pairs of 
unicellular maps. 
An additional novel feature is that our bijection is not {\it always} reducing 
topological genus. In essence we have the following alternative: we either reduce 
genus or we lose connectivity.

 %%%%%%%%%%%%%%%%%%%%%%%%%%%%%%%%%%%%%%%%%%%%%%%%%%
 %%%%%%%%%
 \begin{figure}[ht]
 \begin{center}
 \includegraphics[width=0.7\textwidth]{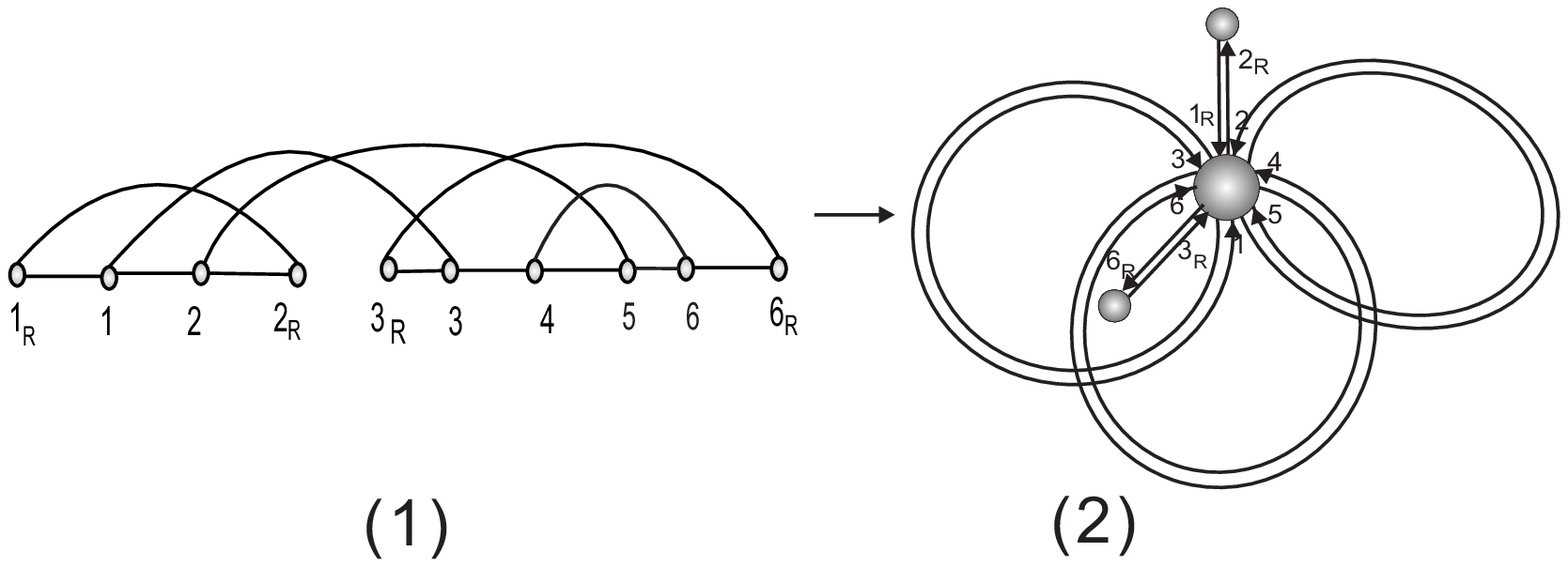}
 \end{center}
 \caption{From a RNA-RNA interaction structures as diagrams to bicellular maps.
 }\label{F:slice1}
 \end{figure}
 %%%%%%%%%
 %%%%
 
 %%%%%%%%%%%%%%%%%%%%%%%%%%%%%%%%%%%%%%%%%%%%%%%%%%
  %%%%%%%%%
  \begin{figure}[ht]
  \begin{center}
  \includegraphics[width=0.7\textwidth]{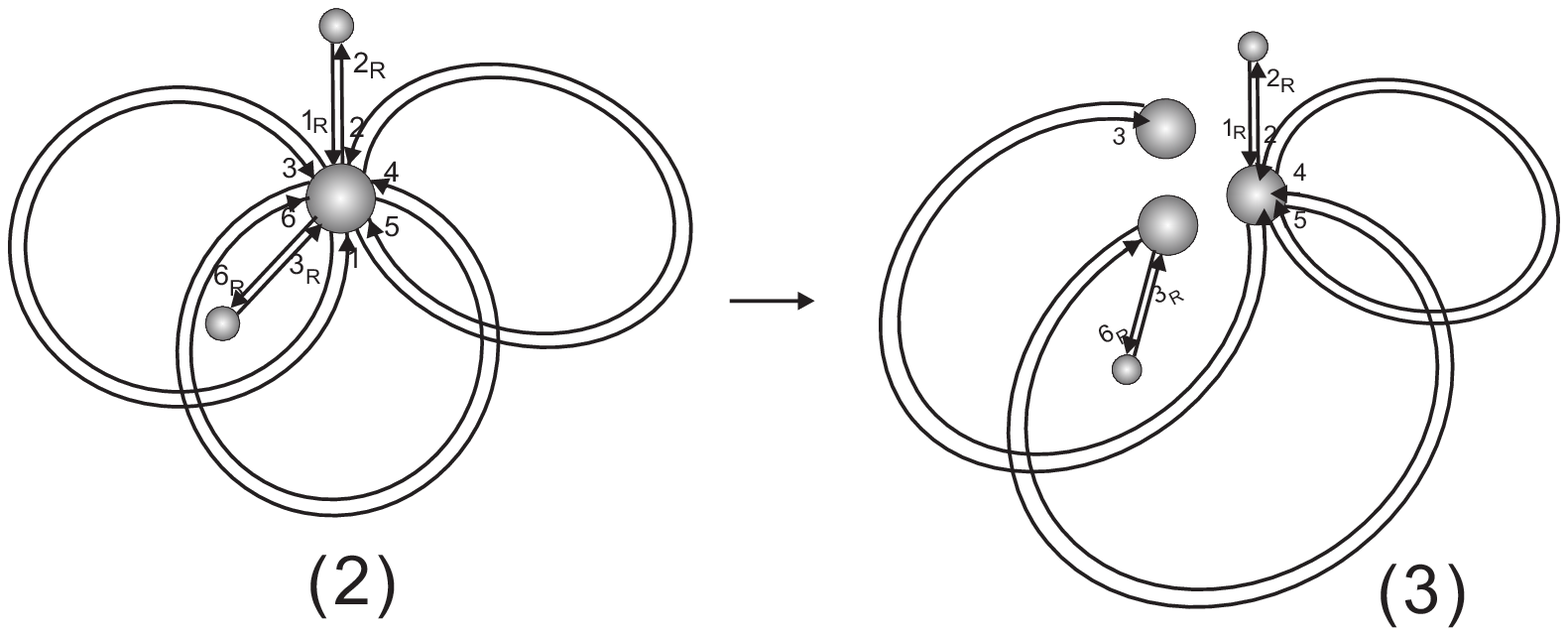}
  \end{center}
  \caption{Slicing bicellular maps, see Section~\ref{S:slgl} for details. 
           Slicing decreases genus by $1$.
  }\label{F:slice2}
  \end{figure}
  %%%%%%%%%
  %%%%
  
  %%%%%%%%%%%%%%%%%%%%%%%%%%%%%%%%%%%%%%%%%%%%%%%%%%
    %%%%%%%%%
    \begin{figure}[H]
    \begin{center}
    \includegraphics[width=0.7\textwidth, height=0.7\textwidth]{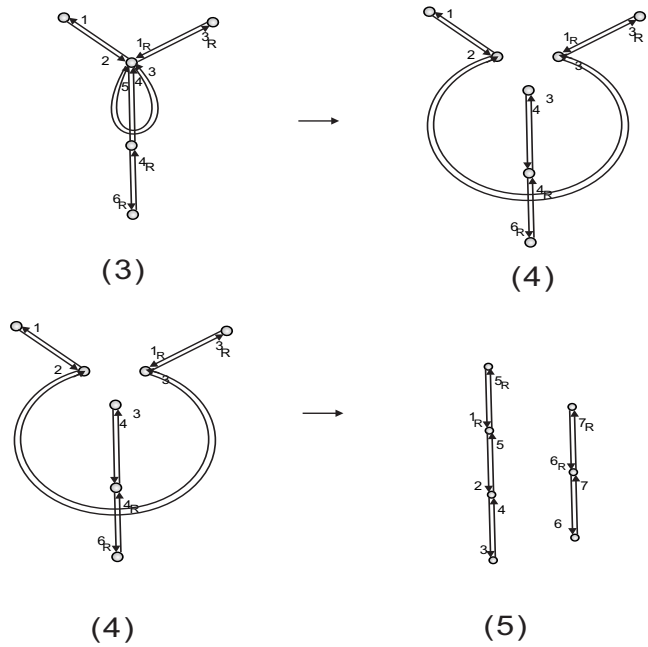}
    \end{center}
    \caption{Further slicing into two plane trees.
    }\label{F:slice3}
    \end{figure}
    %%%%%%%%%
    %%%%
    
The paper is organized as follows: 
 In Section~\ref{S:comp} we show that RNA-complexes are in 
one-to-one correspondence to such maps, namely those that are 
bicellular and planted, see Fig.~\ref{F:slice1}.
This correspondence allows us to perform all our constructions on maps
and eventually recover the diagram thereafter.
In Section~\ref{S:slgl} we study slicing and gluing of bicellular maps. 
We proceed by integrating the results of Section~\ref{S:slgl} into the 
main bijection and its combinatorial corollary in Section~\ref{S:genus}, 
see Fig.~\ref{F:slice2}, \ref{F:slice3}.
Finally we present the uniform generation algorithm in Section~\ref{S:uni}.
Here the idea is to go back, i.e.~we start from a pair of trees and successively
rebuild the bicellular map. Finally, in Section~\ref{S:dis}, we discuss our results 
and show how to use our decomposition grammar to sample RNA-RNA interaction 
structures non-uniformly, employing a simplified loop-base model. This
shows that the unambiguous grammar developed here has many applications and
simply lifts the stochastic-context-free grammar approaches to secondary 
structures to structures with cross-serial interaction arcs.
Various statistics about the loops and stacks in uniformly generated 
complexes of fixed topological genus are given and related to
biological RNA-RNA interaction structures \cite{Richter}.

%%%
%%%%%%%%%%%%%%%%%%%%%%%%%%%%%%%%%%%%%%%%%%%%%%%%%%%%%%%%%%%%%%%%%%%%%%%%%%%
%%%
\section{From RNA-complexes to bicellular maps and back}\label{S:comp}
%%%
%%%%%%%%%%%%%%%%%%%%%%%%%%%%%%%%%%%%%%%%%%%%%%%%%%%%%%%%%%%%%%%%%%%%%%%%%%%
%%%

\begin{definition}
A diagram is a labeled graph over the vertex set $[n]=\{1,2,\ldots,n\}$ 
represented by drawing the vertices $1,2,\ldots,n$ on a horizontal line 
in the natural order and the arcs
$(i,j)$, where $i<j$, in the upper half-plane.
The backbone of a diagram is the sequence of
consecutive integers $(1,\dots,n)$ together with the edges $\{\{i,i+1\}
\mid 1\le i\le n-1\}$.  A diagram over $b$ backbones is a diagram
together with a partition of $[n]$ into $b$ backbones, 
see Fig.~\ref{F:diagram}.
\end{definition}
 %%%%%%%%%%%%%%%%%%%%%%%%%%%%%%%%%%%%%%%%%%%%%%%%%%
 %%%%%%%%%
 \begin{figure}[ht]
 \begin{center}
 \includegraphics[width=0.7\textwidth]{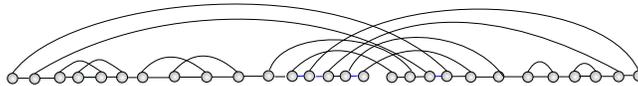}
 \end{center}
 \caption{A $2$-backbone diagram with $28$ vertices and $14$ arcs.
 }\label{F:diagram}
 \end{figure}
 %%%%%%%%%
 %%%%
We shall distinguish backbone edges
$\{i,i+1\}$ from arcs $(i,i+1)$, which we refer to as $1$-arcs.
Two arcs $(i,j)$, $(r,s)$, where $i<r$ are crossing if $i<r<j<s$ holds.
Parallel arcs of the form $\{(i,j), (i+1,j-1), \cdots, (i+\ell-1, j-\ell+1)\}$
is called a stack, and $\ell$ is called the length of this stack.
Furthermore, the particular arc, $(1,n)$, is called the rainbow.

Vertices and arcs of a diagram correspond to nucleotides and base pairs, respectively.
For a diagram over $b$ backbones, the leftmost vertex of each back-bone denotes the
$5'$ end of the RNA sequence, while the rightmost vertex denotes the $3'$ end.
The particular case $b=2$ is referred to as RNA interaction structures. Interaction
structures are oftentimes represented alternatively by drawing the two backbones
on top of each other.

We will add an additional ``rainbow-arc'' over each respective backbone and refer to 
these diagrams as \textit{ planted $2$-backbone diagrams}, see Fig.~\ref{F:planted}. 

 %%%%%%%%%%%%%%%%%%%%%%%%%%%%%%%%%%%%%%%%%%%%%%%%%%
 %%%%%%%%%
 \begin{figure}[ht]
 \begin{center}
 \includegraphics[width=0.7\textwidth]{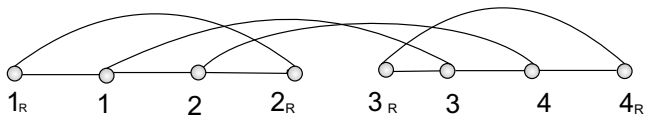}
 \end{center}
 \caption{A planted $2$-backbone diagram with its rainbow arcs
  $(1_R, 2_{R}),(3_R, 4_{R})$.
 }\label{F:planted}
 \end{figure}
 %%%%%%%%%
 %%%%

The specific drawing of a diagram $G$ in the plane determines a cyclic ordering on
the half edges of the underlying graph incident on each vertex, thus defining a
corresponding fat graph $\mathbb{G}$. The collection of cyclic orderings is called
fattening, one such ordering on the half-edges incident on each vertex. Each fat graph
$\mathbb{G}$ determines an oriented surface $F(\mathbb{G})$ which is connected if
$\mathbb{G}$ is and has some associated genus $g(\mathbb{G})\geq 0$ and number 
$r(\mathbb{G})\geq 1$ of boundary
components. Clearly, $F(\mathbb{G})$ contains $G$ as a deformation retract.
Without affecting topological type of the constructed surface, one may collapse each
backbone to a single vertex with the induced fattening called the polygonal model of the
RNA, see Fig~\ref{F:inflation}.
%%%%
%%%%%%%%%%%%%%%%%%%%%%%%%%%%%%%%%%%%%%%%%%%%%%%%%%%%%%%%%%%%%%%%%%%%%%%%%%%%%%%%%%
%%%%%%%%%
 \begin{figure}[ht]
 \begin{center}
 \includegraphics[width=0.9\textwidth]{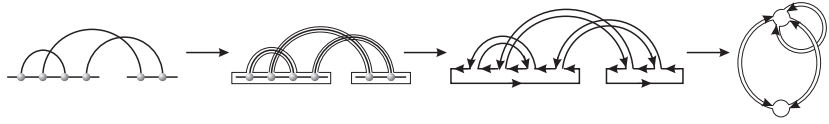}
 \end{center}
 \caption{\small Inflation of a $2$-backbone diagram and collapse of its $2$ backbones
 to two vertices
 }\label{F:inflation}
 \end{figure}
 %%%%
 %%%%%%%%%%%%%%%%%%%%%%%%%%%%%%%%%%%%%%%%%%%%%%%%%%%%%%%%%%%%%%%%%%%%%%%%%%%%%%%%
 %%%%

We next prepare ourselves to study bicellular maps. To this end we discuss the idea behind
general maps:

%%%
%%%%%%%%%%%%%% definition of a map %%%%%%%%%%%%%%%%%%%%%%%%%%%%%%%%%%%%%%%%%%
%%%%
\begin{definition}
Let $n$ be a positive integer. A map of size $n$ is a triple
$\mathfrak{m}=(\gamma, \alpha, \sigma)$ of permutations over
$[1,2n]$ such that: \\
$\bullet$ $\sigma\alpha=\gamma$,\\
$\bullet$ $\alpha$ is  a fixed-point free involution (i.e.~all its cycles
          have length 2).
\end{definition}
%%%%%%%%%%%%%%
%%%%%%%%%%%%%%%%%%%%%%%%%%%%%%%%%%%%%%%%%%%%%%%%%%%%%%%%%%%%%%%%%%%%%%%%%%%%%%
%%%%%
%%%%%
%%%%%%%%%%%%%%%%%%%%%%%%%%%%%%%%%%%%%%%%%%%%%%%%%%%%%%%%%%%%%%%%%%%%%%%%%%%%%%
%%%%%
As usual we write a permutation $\sigma$ as a product of its cycles and
denote the number of its cycles by $|\sigma|$. Suppose we are given a map
$\mathfrak{m}=(\gamma, \alpha, \sigma)$, then the cycles of $\gamma$,
$\alpha$ and $\sigma$ are referred to as \textit{faces}, \textit{edges}, and
\textit{vertices}, respectively.
%%%%%%%%
%%%%%%%%%%%%%%%%%%%% fat graph%%%%%%%%%%%%%%%%%%%%%%%%%%%%%%%%%%%%%%%%%%%%
%%%%

We can use fat graphs $\mathbb{G}$, which sometimes also called ``ribbon graph'',
to give a graphical interpretation of maps. 
A fat graph is a multi-graph (with loops and multiple edges allowed), with
a prescribed cyclic order (counter-clockwise) of the edges around each vertex.

Given a map $\mathfrak{m}=(\gamma,\alpha,\sigma)$, 
its associated fat graph $\mathbb{G}$ is the graph whose edges are 
given by the cycles of $\alpha$, vertices by the cycles of $\sigma$, 
and the natural incidence relation $v\sim e $ if $v$ and $e$ share an element.
 Moreover, we draw each edge of $\mathbb{G}$ as a ribbon, where each side
of the ribbon is called a half-edge; we decide which half-edge corresponds 
to which side of the ribbon by the convention that, 
if a half-edge $h$ belongs to a cycle  $4$ of $\alpha$ and 
$v$ of $\sigma$, then $h$ is the right-hand side of the ribbon 
corresponding to $e$, when considered entering $v$. Furthermore, we draw 
the graph $G$ in such a way that around each vertex $v$, the counter-clockwise
 ordering of the half-edges belonging to the cycle $v$ 
is given by that cycle.  Note that the  cycles of  
of the permutation $\gamma = \sigma\alpha$ are interpreted
as the sequence of half-edges visited when making a tour of 
the graph, keeping the graph on its left, see Fig.~\ref{F:fatgraph}.
%%%
%%%%%%%%%%%%%%%%%%%%%%%%%%%%%%%%%%%%%%%%%%%%%%%%%%%%
%%%
    \begin{figure}[ht]
    \begin{center}
    \includegraphics[width=0.9\textwidth]{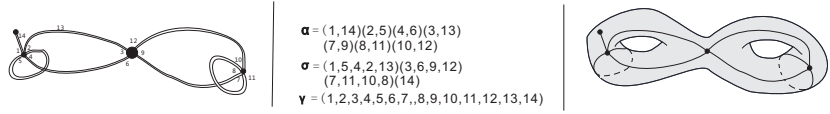}
    \end{center}
    \caption{\small A map $\mathfrak{m}=(\gamma,\alpha,\sigma)$ 
    left: fat graph representation, middle: permutation representation right: 
     topological embedding.  
    }\label{F:fatgraph}
    \end{figure}
%%%
%%%%%%%%%%%%%%%%%%%%%%%%%%%%%%%%%%%%%%%%%%%%%%%%%%
%%%

If the associated fat graph is connected, we call the map connected.
A connected map $\mathfrak{m}$ can be embedded in a compact orientable
surface, such that its complement is a disjoint union of simply 
connected domains (called the faces), and 
considered up to oriented homeomorphism. We can define
the genus $g$ of the map $\mathfrak{m}$ by the genus of the surface. We can
rewrite Euler's characteristic formula in terms of $\sigma,\gamma$ and 
$\alpha$ as $|\sigma|+|\gamma|=|\alpha| + 2-2g$.

%%%
%%%%%%%%%%%%%%%%%%%%%%%%%%%%%%%%%%%%%%%%%%%%%%%%%%%%%%%%%%%%%%%%%%%%%%%%%%%%%%
%%%

Now we are in position to discuss planted, bicellular maps.

%%%
%%%%%%%%%%%%%%%%%%%%%%%%%%%%%%%%%%%%%%%%%%%%%%%%%%%%%%%%%%%%%%%%%%%%%%%%%%%
%%%
\begin{definition}
A map $\mathfrak{b}_g=(\gamma,\alpha,\sigma)$ having $n$ edges, genus $g$ 
and boundary component
$\gamma=\omega_1 \omega_2= (1,2,3,\cdots,k)(k+1,k+2,\cdots,2n)$ is called 
bicellular if there exist some half-edge $x \in \omega_1$, such that $\alpha(x)
\in \omega_2$.
\end{definition}

%%%%%%%%%%%%%%%%%%%%%%%%%%%%%%%%%%%%%%%%%%%%%%%%%%%%%%%%%%%%%%%%%%%%%%%%%%%
%%%

%%%
%%%%%%%%%%%%%%%%%%%%%%%%%%%%%%%%%%%%%%%%%%%%%%%%%%%%%%%%%%%%%%%%%%%%%%%%%%%
%%%
\begin{definition}
A planted, bicellular map $\mathfrak{b}_g$ having $n$ edges and genus $g$
is a bicellular map $\mathfrak{b}_g=(\gamma, \alpha, \sigma)$, such that
$$
\gamma=\omega_1\omega_2=(1_R, 1, \ldots,m, m_R)((m+1)_R, (m+1)
\ldots,  2n,  (2n)_R),
$$
where $|\gamma|=2n+4$ and $\alpha$ is a fixed-point free involution
containing the cycles $(1_R,m_R)$ and $((m+1)_R, 2n_R)$. We refer to the
latter as plants.
\end{definition}
%%%
%%%%%%%%%%%%%%%%%%%%%%%%%%%%%%%%%%%%%%%%%%%%%%%%%%%%%%%%%%%%%%%%%%%%%%%%%%%%
%%%

While bicellular maps are simply particular fat graphs, they naturally
arise as the Poincar\'{e} dual of $2$-backbone diagrams. That is, we have

%%%
%%%%%%%%%%%%%%%%%%%%%%%%%%%%%%%%%%%%%%%%%%%%%%%%%%%%%%%%%%%%%%%%%%%%%%%%%%%%
%%%
\begin{lemma} \label{L:diagram}
There is a bijection between planted $2$-backbone diagrams 
and planted bicellular maps.
\end{lemma}
%%%
%%%%%%%%%%%%%%%%%%%%%%%%%%%%%%%%%%%%%%%%%%%%%%%%%%%%%%%%%%%%%%%%%%%%%%%%%%%%
%%%
\begin{proof}
Given a planted $2$-backbone diagram, we inflate the arcs, collapse each backbone 
into a single vertex, see Fig.~\ref{F:inflation}. This produces a fat graph with 
two vertices $(\gamma, \alpha, \sigma)$. Next we consider the mapping
(note $\gamma=\alpha\circ \sigma$):
$$
\pi\colon (\gamma, \alpha, \sigma ) \rightarrow
(\sigma, \alpha,\alpha\circ \sigma).
$$
$\pi$ is evidently a bijection between fat graphs having two vertices and
bicellular maps, see Fig.~\ref{F:dual}. The mapping is an instantiation
of the Poincar\'{e} dual and interchanges boundary components with vertices,
preserving by construction topological genus.
\end{proof}

 %%%%%%%%%
 \begin{figure}[ht]
 \begin{center}
 \includegraphics[width=0.9\textwidth]{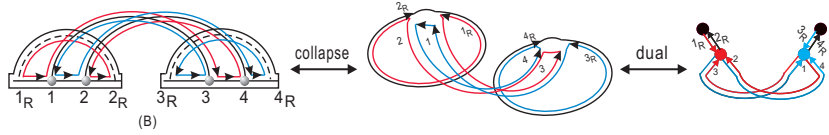}
 \end{center}
 \caption{\small The Poincar\'{e} dual: from RNA complexes to bicellular maps and
          back.}\label{F:dual}
 \end{figure}
 %%%%%%%%%
 %%%%
%%%
%%%%%%%%%%%%%%%%%%%%%%%%%%%%%%%%%%%%%%%%%%%%%%%%%%%%%%%%%%%%%%%%%%%%%%%%%%%
%%%
%%%
%%%%%%%%%%%%%%%%%%%%%%%%%%%%%%%%%%%%%%%%%%%%%%%%%%%%%%%%%%%%%%%%
%%%
\section{Slicing and gluing in bicellular maps}\label{S:slgl}
%%%
%%%%%%%%%%%%%%%%%%%%%%%%%%%%%%%%%%%%%%%%%%%%%%%%%%%%%%%%%%%%%%%%
%%%

Given a bicellular map $\mathfrak{b}_g$, the permutations $\sigma$ and $\gamma$
induce the following two linear orders $<_{\gamma}$ and $<_{\sigma}$ of half-edges:
To define $<_\gamma$, we set $r_1<_{\gamma} r_2$ for $r_1 \in \omega_1$ and $r_2 \in
\omega_2$ and
$$
r <_{\gamma} \omega_i(r) <_{\gamma} \dots <_{\gamma} \omega_i^{k-1}(r), \qquad
\text{for} \ \  i\in \{1,2\}.
$$
Note that the minimal element here is the half-edge coming out from the first plant.

 \begin{figure}[ht]
 \begin{center}
 \includegraphics[width=0.4\textwidth]{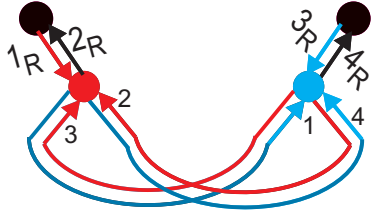}
 \end{center}
 \caption{\small The two orders: 
 $1_R<_{\gamma} 1<_{\gamma} 2<_{\gamma} 2_R <_{\gamma}3_R <_{\gamma}3 <_{\gamma} 4 <_{\gamma} 4_R$
 and $1_r<_{\sigma} 3<_{\sigma}2$, $1<_{\sigma} 4 <_{\sigma} 3_R$.
 }\label{F:order}
 \end{figure}

In order to define $<_\sigma$ we set for any vertex $v=\sigma_i$:
$$
r <_{\sigma} \sigma_i(r) <_{\sigma} \dots <_{\sigma} \sigma_i^{k}(r).
$$
Let $a_1$ and $a_2$ be two distinct half-edges in $\mathfrak{b}_g$. Then
$a_1<_{\gamma} a_2$ expresses the fact that $a_1$ appears before $a_2$ in
the boundary component $\omega_i$ or $a_1 \in \omega_1 $ and $a_2 \in \omega_2 $.
Suppose two half-edges $a_1$ and $a_2$ belong to the same vertex $v$. Note
that $v$ is a cycle which we assume to originate with the first
half-edge along which one enters $v$ travelling $\gamma$. Then
$a_1<_\sigma a_2$ expresses the fact that $a_1$ appears (counter-clockwise)
before $a_2$.

%%%
%%%%%%%The definition of trisection%%%%%%%%%%%%%%%%%%%%%%%%%%%%%%%%%%%%%%%%%
%%%
\begin{definition}
A half-edge $h$ is an {\it up-step} if $h<_{\gamma} \sigma(h)$, and a
{\it down-step} if $\sigma(h) \le_{\gamma} h$. $h$ is called a
{\it trisection} if $h$ is a down-step and $\sigma(h)$ is not the
minimum half-edge of its respective vertex.
\end{definition}
%%%
%%%%%%%%%%%%%%%%%%%%%%%%%%%%%%%%%%%%%%%%%%%%%%%%%%%%%%%%%%%%%%%%
%%%

This following lemma is the analogon to the trisection lemma of \citep{Chapuy:11}
for planted bicellular maps:

%%%
%%%%%%%%%%%%%%%%%%%%%%%%%%%%%%%%%%%%%%%%%%%%%%%%%%%%%%%%%%%%%%%%%%%%%%%%%%%%
%%%
\begin{lemma}\label{L:trisection}
Any planted bicellular map, $\mathfrak{b}_g=(\gamma, \alpha, \sigma)$,
has $2(g+1)$ trisections.
\end{lemma}
%%%
%%%%%%%%%%%%%%%%%%%%%%%%%%%%%%%%%%%%%%%%%%%%%%%%%%%%%%%%%%%%%%%%%%%%%%%%%%%%
%%%
\begin{proof}
Let $n_+$ and $n_-$ denote the number of up-steps and down-steps in
$\mathfrak{b}_g$. Then we have $n_+ + n_-=2n+4$, where $n$ is the number
of edges of $\mathfrak{b}_g$. Let $i$ be a half-edge of $\mathfrak{b}_g$,
and $j=\sigma^{-1}\alpha\sigma(i)$. Observe that we have $\sigma(j)=
\gamma(i)$, and $\gamma(j)=\sigma(i)$. It is clear
that if the tour of the map visits $i$ before $\sigma(i)$, then it
necessarily visits $\sigma(j)$ before $j$, see Fig.~\ref{F:1}.
We distinguish four cases:

%%%
%%%%%%%%%%%%%%%%%%%%%%%%%%%%%%%%%%%%%%%%%%%%%%%%%%%%%%%%%
%%%
\begin{figure}[ht]
\begin{center}
\includegraphics[width=4cm]{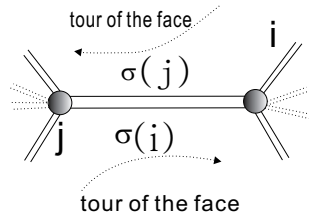}
\end{center}
\caption{The main argument in the proof of the trisection
lemma: the tour of the face visits $i$ before $\sigma(i)$ if
and only if it visits $\sigma(j)$ before $j$, unless
$\sigma(i)$ or $\sigma(j)$ is
the plants of the bicellular map.}
\label{F:1}
\end{figure}
%%%
%%%%%%%%%%%%%%%%%%%%%%%%%%%%%%%%%%%%%%%%%%%%%%%%%%%%%%%
%%%
First suppose $i<_{\gamma} \sigma(i)=\gamma(j)$, i.e.~$i$ is an up-step.
Then $i<_{\gamma} \gamma(j)$ implies $\gamma(i)\leq_{\gamma}\gamma(j)$.
Since $\gamma(i)=(\sigma\alpha)(i)=\sigma(j)$, this means $\sigma(j)
\leq_{\gamma} \gamma(j)$. By definition of $<_{\gamma}$, this implies
$\sigma(j)\leq_{\gamma} j$ since $j$ is maximal with this property and
$\sigma(j) \neq \gamma(j)$ ($\alpha$ has no fixed point).
Accordingly, if $i$ is an up-step, then $j$ is a down-step.

Second, assume that $\sigma(i)=\gamma(j)\le_\gamma i$ and that
$\gamma(j)$ is not one of the two plants.
In this case, $j<_{\gamma} \gamma(j)$ implies
$$
j<_\gamma\gamma(j)=\sigma(i) \leq_{\gamma} i <_{\gamma} \gamma(i) =\sigma(j),
$$
that is, $j<_{\gamma} \sigma(j)$, and $j$ is consequently an up-step.

Third suppose that $\sigma(i)=\gamma(j)\le_\gamma i$ and $\gamma(j)=1_R$.
Then $j = m_R= \alpha (1_R)= \alpha (\gamma (j))=\sigma(j)$ and
$j$ is a down-step.

Fourth we suppose $i$ is a down-step and $\gamma(j)= (m+1)_R$.
Then $j=2n_R$ is the biggest label of the half-edges. So $j$ is
always a down-step, i.e.~$j\geq_{\gamma} \sigma(j)$.

Therefore we have proved that each edge, except of $(1_R, m_R)$ and
$((m+1)_R,2n_R)$ is associated to one up-step and one down-step.
As a result there are exactly four more down-steps than up-steps,
i.e. $n_- =n_+ + 4$, whence $n_-=n+4$.

Since each vertex carries exactly one down-step which is not a
trisection (its minimal half-edge), the total number of trisections
equals $(n_- - v)$, where $v$ is the number of vertices of
$\mathfrak{b}_g$. Euler's characteristic formula, $v = n + 2 - 2g$,
implies that the number of trisections is $n+4-(n+2-2g)=2g+2$, whence the
lemma. 
\end{proof}

We next study the effect of slicing and gluing in bicellular maps.
%%%
%%%%%%%%%% Slicing and gluing in bicellular maps %%%%%%%%%%%%%%%%%%%%%%%%%%%
%%%
\begin{lemma}{\bf (slicing and gluing)}\label{L:2faces}
Suppose
$$
v:=(a_1, h_2^1, \ldots,  h_2^{m_2}, a_2, h_3^1, \ldots, h_3^{m_3}, a_3, h_1^1,
 \ldots, h_1^{m_1} )
$$
is a $\mathfrak{b}_g$-vertex and $a_1,a_2,a_3$ are intertwined,
i.e.~$a_1<_{\gamma} a_3<_{\gamma} a_2$ and $a_1 <_{\sigma} a_2 <_{\sigma} a_3$.
Then slicing $v$ via $\{a_1,a_2,a_3\}$ produces either a bicellular map
$\mathfrak{b}_{g-1}$, or a pair of unicellular maps
$(\mathfrak{u}_{g_1},\mathfrak{u}_{g-g_1})$.
Furthermore, slicing can be reversed by gluing via $\{a_1,a_2,a_3\}$.
\end{lemma}
%%%
%%%%%%%%%% Slicing and gluing in bicellular maps %%%%%%%%%%%%%%%%%%%%%%%%%%%%
%%%
\begin{proof}
Suppose
$$
\gamma=\omega_1\omega_2=((1_{R_b},1_b,\dots, m_b, m_{R_b})
((m+1)_{R_b},(m+1)_b, \dots, (2n)_b, (2n)_{R_b})).
$$
We distinguish the following two scenarios:\\
%%%
%%%%%%%%%%%%%%%%%%%%%%%%%%%%%%%%%%%%%%%%%%%%%%%%%%%%%%%%%%%%%%%%%%%%%%%%%%%%%%
%%%
{\it Case $1$.} $a_1,a_2,a_3$ are either contained in $\omega_1$ or $\omega_2$.\\
%%%
%%%%%%%%%%%%%%%%%%%%%%%%%%%%%%%%%%%%%%%%%%%%%%%%%%%%%%%%%%%%%%%%%%%%%%%%%%%%%%
%%%
In this case it is clear that slicing preserves bicellularity. Indeed,
suppose $\{a_1,a_2,a_3\} \in \omega_1$, then we can write the two faces
as
$$
(a_1, k_2^1,k_2^2,\cdots,k_2^{l_2},a_3,k_1^1,k_1^2,\cdots,k_1^{l_1}, a_2,k_3^1,k_3^2,
\cdots,k_3^{l_3})(k_4^1,k_4^2,\cdots, k_4^{l_4} )
$$
and slicing generates the boundary components
\begin{equation*}
\begin{split}
\hat{\gamma}=\hat{\omega}_1\hat{\omega}_2= &
(a_1,k_1^1,k_1^2,\cdots,k_1^{l_1}, a_2, k_2^1,k_2^2,
\cdots,k_2^{l_2},a_3,k_3^1,k_3^2,\cdots,k_3^{l_3})\\
&\cdot(k_4^1,k_4^2,\cdots, k_4^{l_4} ).
\end{split}
\end{equation*}
Hence $E_{con}:=\{(x, \alpha(x)),| x\in \omega_1, \alpha(x)\in \omega_2\}$ is
unaffected by slicing, which implies that the sliced map remains bicellular
having two additional vertices. Since the number of edges remains constant
Euler characteristic implies that the genus decreases by $1$, see Fig.~\ref{F:case1}.
%%%
%%%%%%%%%%%%%%%%%%%%%%%%%%%%%%%%%%%%%%%%%%%%%%%%%%%%%%%%%
%%%
\begin{figure}[ht]
\begin{center}
\includegraphics[width=0.8\textwidth]{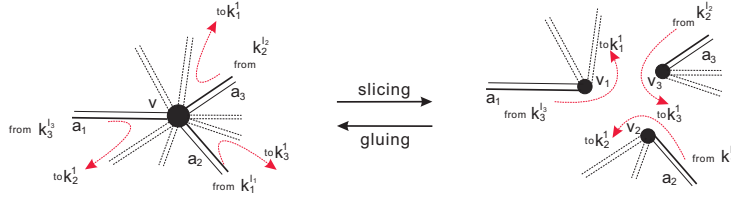}
\end{center}
\caption{Gluing and slicing: case $1$}
\label{F:case1}
\end{figure}
%%%
%%%%%%%%%%%%%%%%%%%%%%%%%%%%%%%%%%%%%%%%%%%%%%%%%%%%%%%
%%%
%%%
%%%%%%%%%%%%%%%%%%%%%%%%%%%%%%%%%%%%%%%%%%%%%%%%%%%%%%%%%%%%%%%%%%%%%%%%%%%%%%
%%%
{\it Case $2$.} 
$a_1,a_2,a_3$ are not contained in either one of the boundary components.
%%%
%%%%%%%%%%%%%%%%%%%%%%%%%%%%%%%%%%%%%%%%%%%%%%%%%%%%%%%%%%%%%%%%%%%%%%%%%%%%%%
%%%
Clearly, $a_1<_{\gamma} a_3<_{\gamma}a_2$ implies that we have the alternative
$\{a_1,a_3\} \in \omega_1, a_2 \in {\omega}_2$, or  $\{a_1\} \in \omega_1,
\{a_3,a_2\} \in {\omega}_2$.

In case of $\{a_1,a_3\} \in \omega_1, a_2 \in {\omega}_2$, we rewrite the two
faces as
$$
\gamma = \omega_1\omega_2=(a_1,k_2^1,k_2^2,\cdots,k_2^{l_2}, a_3,k_1^1,k_1^2,
\cdots,k_1^{l_1})(a_2,k_3^1,\cdots k_3^{l_3}).
$$
We next consider the half-edges whose image is not the same for $\gamma$ and
$\hat{\gamma}$. These are $\{a_1,a_2,a_3\}$ and by construction we
have
$$
\hat{\gamma}(a_1)=\alpha \hat{\sigma}(a_1)=\alpha(h_1)=\alpha(\sigma(a_3))=
\gamma(a_3)=k_1^1.
$$
Similarly, we have $\hat{\gamma}(a_2)=k_2^1$, and $\hat{\gamma}(a_3)=k_3^1$. Thus we
arrive at
$$
\hat{\gamma}=\hat{\omega}_1\hat{\omega}_2=
(a_1,k_1^1,k_1^2,\cdots,k_1^{l_1})(a_2,k_2^1,k_2^2,\cdots,
 k_2^{l_2},a_3,k_3^1,k_3^2,\cdots,k_3^{l_3}).
$$
Accordingly, slicing maps the set of half-edges
$E_{mov}=\{k_2^1,k_2^2,\cdots,k_2^{l_2}, a_3\}$
into the second boundary component, $\omega_2$. If
$E_{con}\nsubseteq E_{mov} $, then we still have a bicellular map with genus $(g-1)$.
However, if $E_{con}\subseteq E_{mov}$, then slicing produces a pair of unicellular
maps $(\mathfrak{u}_1,\mathfrak{u}_2)$.
Suppose that in this case $\mathfrak{u}_1$ has $v_1$ vertices, $m_1$ edges, and genus
$2-2g_1=v_1+1-m_1$. Then $\mathfrak{u}_2$ has $v_2=v-v_1+2$ vertices, $m_2=m-m_1$ edges and
genus $2-2g_2=v_2+1-m_2$, whence $g_2=g-g_1$, see Fig.~\ref{F:case2}.
%%%
%%%%%%%%%%%%%%%%%%%%%%%%%%%%%%%%%%%%%%%%%%%%%%%%%%%%%%%%%
%%%
\begin{figure}[ht]
\begin{center}
\includegraphics[width=0.8\textwidth]{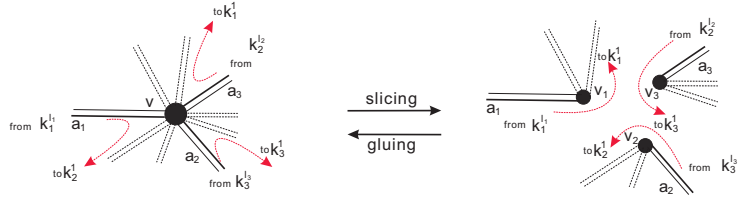}
\end{center}
\caption{Gluing and slicing: case $2$}
\label{F:case2}
\end{figure}
%%%
%%%%%%%%%%%%%%%%%%%%%%%%%%%%%%%%%%%%%%%%%%%%%%%%%%%%%%%
%%%

The case $\{a_1\} \in \omega_1, \{a_3,a_2\} \in {\omega}_2$ is treated analogously.
It is straightforward to verify that given $\{a_1,a_2,a_3\}$, slicing can be reversed by
gluing. 
\end{proof}

%%%
%%%%%%%%%%%%%%%%%%%%%%%%%%%%%%%%%%%%%%%%%%%%%%%%%%%%%%%%%%%%%%%%%%%%%%%%%%%%%%
%%%
\begin{definition}\label{D:slicing}
Given a bicellular map and a distinguished trisection $\tau$ at vertex $v$,
we set $a_1$ to be the minimum half-edge of $v$, $a_3$ the half-edge
following $\tau$ counter-clockwise and $a_2$ to be the smallest half-edge on
the left of $a_3$, which is greater than $a_3$.
Then slicing $v$ via $\{a_1,a_2,a_3\}$, we have two scenarios:
either $a_3$ is the minimum half-edge of its respective vertex, or not.
In the former case $\tau$ is called type $I$ and type $II$ in the latter.
\end{definition}
%%%
%%%%%%%%%%%%%%%%%%%%%%%%%%%%%%%%%%%%%%%%%%%%%%%%%%%%%%%%%%%%%%%%%%%%%%%%%%%%%%
%%%

Let $D_{b,g+1}^{I}(n)$ denote the set of bicellular maps of genus $(g+1)$ with $n$
edges and a distinguished trisection of type $I$. Let $B_g^3(n)$ denote the set
of bicellular maps of genus $g$ with $n$ edges and three distinguished
vertices.
Finally, let $(U_{g_1},U_{g+1-g_1})^3$ denote the set of pairs of unicellular maps
whose sum of genera equals $(g+1)$, having combined $n$ edges, such that
not all distinguished vertices are contained exclusively in one map. We call
such vertex configurations distributed.

%%%
%%%%%%%%%%%%%%%%%%%%%%%%%%%%%%%%%%%%%%%%%%%%%%%%%%%%%%%%%%%%%%%%%%%%%%%%%%%%%%
%%%
\begin{lemma}{\label{L:phi}}
There is a bijection
$$
\phi\colon D_{b,g+1}^{I}(n) \rightarrow B_g^3(n) \dot\bigcup  
(U_{g_1},U_{g+1-g_1})^3.
$$
\end{lemma}
%%%
%%%%%%%%%%%%%%%%%%%%%%%%%%%%%%%%%%%%%%%%%%%%%%%%%%%%%%%%%%%%%%%%%%%%%%%%%%%%%%
%%%
\begin{proof}
Given a bicellular map $\mathfrak{b}_{g+1}(n)$ with vertex $v$, having a
distinguished
type $I$ trisection, $\tau$. Let $\phi$ be the slicing map of $v$ via the
half-edge set $C_{\tau}=\{a_1,a_2,a_3\}$, as in Def.~\ref{D:slicing}.
By construction, we have $a_1<_{\gamma} a_3<_{\gamma} a_2$, i.e.~the half-edges
are intertwined.
From Lemma~\ref{L:2faces} we know that slicing produces either a bicellular
map, $\mathfrak{b}_{g-1}(n)$, or alternatively a pair of unicellular maps 
$(\mathfrak{u}_{g_1},\mathfrak{u}_{g+1-g_1})$.
Furthermore slicing produces a triple $(v_1,v_2,v_3)$, such that $a_i$ is the
minimum in the vertex $v_i$, respectively. In case of $(\mathfrak{u}_{g_1},
\mathfrak{u}_{g+1-g_1})$, such vertices are distributed. Accordingly, $\phi$ is
well-defined.

We proceed by constructing the inverse of $\phi$. To this end, let $\mathfrak{b}_{g}$
be a bicellular map of genus $g$ with three distinguished vertices
$\{v_1,v_2,v_3\}$, where
$$
a_1=\min_{\gamma} v_1<_\gamma a_2=\min_{\gamma} v_2 <_\gamma a_3=\min_{\gamma} v_3.
$$
Let $\chi(\mathfrak{b}_g)$ be the map obtained by the gluing of $\mathfrak{b}_{g}$
via $\{a_1,a_2,a_3\}$.
By construction, $a_1$ remains to be the minimum half-edge of the vertex.
$\sigma^{-1} (a_3)$  becomes a trisection which, by construction, is of type
$I$ and $a_2$ is by construction the smallest half-edge to the left of $a_3$
that is larger than $a_3$.
Similarly, suppose we are given a pair $(\mathfrak{u}_{g_1},\mathfrak{u}_{g+1-g_1})$ with
three distinguished, distributed vertices $v_1,v_2,v_3$. By construction, gluing
produces a bicellular map $\chi(\mathfrak{b}_g)$ with a distinguished trisection,
$\sigma^{-1}(a_3)$. As slicing and gluing are inverse operations we have
$\chi\circ \phi=\text{\rm id}$ and $\phi\circ \chi=\text{\rm id}$,
whence $\phi$ is a bijection. 
\end{proof}

Let $D_{b,g+1}^{II}(n)$ denote the set of bicellular maps of genus $(g+1)$ with
$n$ edges and a distinguished trisection of type $II$. Let $\nu_{\mathfrak{b},g}(n)$ be
the set of $4$-tuples $(\mathfrak{b}_g,v_1,v_2,\tau)$, where $\mathfrak{b}_g$ is
a bicellular map of genus $g$ with $n$ edges and where $v_1,v_2$ and $\tau$ are
two vertices and a trisection of $\mathfrak{b}_g$ such that:
$$
\min_{\gamma} v_1 <_{\gamma} \min_{\gamma} v_2 <_{\gamma} \min_{\gamma} v(\tau).
$$
Let $\kappa_{g+1}(n)$ be the set of  $5$-tuples
$(\mathfrak{u}_{g_1},\mathfrak{u}_{g+1-g_1}, v_1, v_2, \tau)$, where $\mathfrak{u}_{g_1}$
and $\mathfrak{u}_{g+1-g_1}$ are unicellular maps, with genus $g_1$ and $g+1-g_1$.
Furthermore, $v_1,v_2$ and $\tau$ are distributed, i.e.~not all three are contained in
$\mathfrak{u}_{g_1}$ or $\mathfrak{u}_{g+1-g_1}$.

Then we have the following analogon of Lemma~\ref{L:phi}.
%%%
%%%%%%%%%%%%%%%%%%%%%%%%%%%%%%%%%%%%%%%%%%%%%%%%%%%%%%%%%%%%%%%%%%%%%%%%%%%%%%%%%%%%%%%%%%%
%%%
\begin{lemma}{\label{L:psi}}
There exists a bijection $\psi\colon D_{b,g+1}^{II}(n) \rightarrow \nu_{b,g}(n) \dot\bigcup
\kappa_{g+1}(n)$.
\end{lemma}
%%%
%%%%%%%%%%%%%%%%%%%%%%%%%%%%%%%%%%%%%%%%%%%%%%%%%%%%%%%%%%%%%%%%%%%%%%%%%%%%%%%%%%%%%%%%%%
%%%
\begin{proof}
Let $\mathfrak{b}_{g+1}(n)$ be a bicellular map of genus $(g+1)$ with vertex $v$ and
distinguished type $II$ trisection $\tau$. Let $\psi$ be the slicing of
$v$ via $C=\{a_1,a_2,a_3\}$. Where the $a_i$ are chosen as in
Lemma~\ref{L:phi}. Lemma~\ref{L:2faces} guarantees that $\psi$ generates
either a bicellular map or a pair of unicellular maps and in the latter case
$v_1, v_2$ and $\tau$ are distributed.
However, slicing does not render $a_3$ as the minimum of $v_3$, since the
trisection is type $II$. In fact, $\tau$ is again a trisection of $v_3$,
since, by construction,
$$
\hat{\sigma}(\tau)=a_3 \ \text{\rm and} \ a_3<_{\hat{\gamma}}\tau
$$
and there exist some $h_i \in \{h_3^1,h_3^2,\cdots h_3^{l_3}\}$ such that
$h_i <_{\hat{\gamma}} a_3$. Therefore
$$
\psi\colon D_{b,g+1}^{II}(n) \rightarrow \nu_{b,g}(n) \dot\bigcup  \kappa_{g+1}(n)
$$
is well-defined.

We proceed by specifying the inverse of $\phi$, $\chi$. Suppose we are given
a bicellular map $\mathfrak{b}_g(n)$ or a pair of unicellular maps
$(\mathfrak{u}_{g_1},\mathfrak{u}_{g+1-g_1}) \in
\kappa(n)$ ) with two vertices $v_1, v_2$ and a trisection $\tau$. In case
of $(\mathfrak{u}_{g_1},\mathfrak{u}_{g+1-g_1})$ $v_1, v_2$ and $\tau$ are
distributed. $v_i$ and
$\tau$ have the property
$$
\min_{\hat{\gamma}}{v_1} <_{\hat{\gamma}} \min_{\hat{\gamma}} v_2 <_{\hat{\gamma}}
\min_{\hat{\gamma}} V(\tau).
$$
Then we glue via the half-edges $a_1=\min_{\hat{\gamma}}{v_1},
a_2=\min_{\hat{\gamma}} v_2$ and $a_3=\hat{\sigma}(\tau)$.
By construction, $a_3$ is not minimal at $v$, whence $\tau$ is, after gluing,
a type $II$ trisection. Lemma~\ref{L:2faces} shows that the image of this
gluing is contained in
$D_{b,g+1}^{II}(n)$, whence
$$
\chi\colon \nu_{\mathfrak{b},g}(n) \dot\bigcup  \kappa_{g+1}(n)
\rightarrow D_{b,g+1}^{II}(n)
$$
is well-defined. By construction we have
$\psi\circ \chi=\text{\rm id}$ and $\chi\circ \psi=\text{\rm id}$
and $\psi$ is a bijection. 
\end{proof}

%%%
%%%%%%%%%%%%%%%%%%%%%%%%%%%%%%%%%%%%%%%%%%%%%%%%%%%%%%%%%%%%%%%%%%%%%%%%%%%
%%%
\section{The bijection}\label{S:genus}
%%%
%%%%%%%%%%%%%%%%%%%%%%%%%%%%%%%%%%%%%%%%%%%%%%%%%%%%%%%%%%%%%%%%%%%%%%%%%%%
%%%

Let $D_{b,g}(n)=D_{b,g}^{I}(n)\cup D_{b,g}^{II}(n)$ be the set of bicellular maps of
genus $g$ with $n$ edges and a distinguished trisection. Let $B_g(n)^{t}$ denote
the set of a bicellular map of genus $g$ with $n$ edges and $t$ distinguished
vertices. Finally, let $(U_{g_1}(m), U_{g-g_1}(n-m))^{t}$, for $0\leq  g_1 \leq g$
denote the set of pairs of unicellular maps $(\mathfrak{u}_{g_1}(m),
\mathfrak{u}_{g-g_1}(n-m))$ with $m$ and $n-m$ edges and $t$ distinguished,
distributed vertices.

%%%
%%%%%%%%%%%%%%%%%%%%%%%%%%%%%%%%%%%%%%%%%%%%%%%%%%%%%%%%%%%%%%%%%%%%%%%%%%%%%%%%%%%%%%%%%%
%%%
\begin{theorem}\label{T:main}
There exists a bijection
\begin{equation}\label{E:B-slicing}
\begin{split}
&\quad
\Xi_b \colon D_{b,g}(n) \rightarrow \\
& \left({\dot\cup}_{p=0}^{g-1} B_{p}(n)^{2g-2p+1} \right)
  {\dot\bigcup} \, \left( {\dot\cup}_{g_1=0}^{p} {\dot\cup}_{m=0}^{n}{\dot\cup}_{p=0}^{g}
\left( U_{g_1 }(m), U_{p-g_1} (n-m )\right)^{2g-2p+3} \right).
\end{split}
\end{equation}
\end{theorem}
%%%
%%%%%%%%%%%%%%%%%%%%%%%%%%%%%%%%%%%%%%%%%%%%%%%%%%%%%%%%%%%%%%%%%%%%%%%%%%%%%%%%%%%%%%%%%%
%%%
\begin{proof}
Suppose we are given a bicellular map $\mathfrak{b} \in D_{b,g}(n)$ with distinguished
trisection $\tau$. Then we can recursively slice $\tau$, as long as it
remains a type $II$ trisection. Clearly, $\tau$ must, after a finite number
of slicings, become of type $I$ and one more slicing resolves the latter
into three distinguished vertices.
In case of slicing into a pair of unicellular maps the distinguished vertices
are distributed.
Each slicing of a type $II$ trisection produces vertices
$$
\min_{\gamma}v_1 <_{\gamma} \min_{\gamma } v_2 <_{\gamma} \min_{\gamma} v_{\tau},
$$
and we write as $v_1<v_2<\tau$ for short. Since slicing does
not affect the order of the half-edges between the plant and the
minimum half-edge of the triple $\{a_1,a_2,a_3\}$,
iterated slicings produces a sequence $v_1<v_2<\cdots< v_{2g-2p+1}$.

According to Lemma~\ref{L:phi} and Lemma~\ref{L:psi}, the slicing of trisections of
type $II$ and $I$ are indeed bijections. Furthermore, every slicing
decreases topological genus by exactly $1$ or $0$(lose connectivity).
 As a result, after iteratively
slicing the type II trisections, we obtain a type $I$ trisection. Then
one more slicing generates an element of
$$
\left({\dot\cup}_{p=0}^{g-1} B_{p}(n)^{2g-2p+1} \right)
  {\dot\bigcup} \, \left( {\dot\cup}_{g_1=0}^{p}
 {\dot\cup}_{m=0}^{n}{\dot\cup}_{p=0}^{g}
\left( U_{g_1 }(m), U_{p-g_1} (n-m )\right)^{2g-2p+3} \right)
$$
and $\Xi_b$ is accordingly well-defined. Clearly, $\Xi_b$ is as the
composition of bijections a bijection. 
\end{proof}
%%%

Let ${\sf U}_g(n)$ and ${\sf B}_g(n)$ denote the number of unicellular maps with genus
$g$ and $n$ half-edges. Using the trisection lemma, Euler's formula, and
Theorem~\ref{T:main}, we obtain the following identity:

%%%
%%%%%%%%%%%%%%%%%%%%%%%%%%%%%%%%%%%%%%%%%%%%%%%%%%%%%%%%%%%%%%%%%%%%%%%%%%%%%%%
%%%
\begin{corollary}\label{T:Rec-Bi}
\begin{equation}\label{E:2g+2}
\begin{split}
&2(g+1){\sf B}_g(n)= \sum_{i=0}^{g-1} {n-2i  \choose 2g-2i+1 } \, {\sf B}_{i}(n) \,  +\\
&  \sum_{i=0}^{g}  \sum_{g_1=0}^{i} \sum_{m = 0}^{n}
\sum_{k =1}^{2g-2i+2}\left({\mu_{g_1} \choose k} {\nu_{g_1,i}\choose 2g-2i+3-k} {\sf U}_{g_1}(m)
{\sf U}_{i-g_1}(n-m)\right)
\end{split}
\end{equation}
where
\begin{eqnarray*}
\mu_{g_1}        & = & m+1-2g_1 \\
\nu_{g_1,i}     & = & n-m -2i +2g_1+1.
\end{eqnarray*}
\end{corollary}
%%%
%%%%%%%%%%%%%%%%%%%%%%%%%%%%%%%%%%%%%%%%%%%%%%%%%%%%%%%%%%%%%%%%%%%%%%%%%%%%%%%
%%%
\begin{proof}
Consider $\mathfrak{b}_{i}$ with $n$ arcs. According to Euler's formula,
we have $v=n-2i$ vertices. Since every gluing increases genus by $1$, there are
exactly $(g-i)$ gluings to derive genus $g$. The first gluing requires $3$ vertices
and generates a type $I$ trisection. Every following step requires $2$ vertices,
whence we choose $2(g-i)+1$ vertices. This interprets the binomial coefficients
${n-2i  \choose 2g-2i+1 }$. Similarly, if we are given a pair
$(\mathfrak{u}_{g_1},\mathfrak{u}_{i-g_1})$,
Suppose $\mathfrak{u}_{g_1}$ has $m$ edges and the other map has $(n-m)$ arcs.
Then $\mathfrak{u}_{g_1}$  and $\mathfrak{u}_{i-g_1}$ have $\mu_{g_1}= m-2g_1+1$, 
$\nu_{g_1,i}=n-m-2(i-g_1)+1$ vertices, respectively.
Every gluing step increases genus by $1$ except one step which 
connects the two unicellular maps to a bicellular map, 
preserving the genus. Thus $(g-i+1)$ gluings generate genus $g$ and 
we need to choose $2(g-i+1) +1$ distributed
vertices. Suppose we choose $k$ vertices on
$\mathfrak{u}_{g_1}$, then $k$ must satisfy
$1\leq k \leq 2(g-i+1)$ and the other vertices
 are selected from $\mathfrak{u}_{i-g_1}$.
This interprets the binomial coefficients
 ${\mu_{g_1} \choose k} {\nu_{g_1,i}\choose
2g-2i+3-k}$, whence the corollary. 
\end{proof}

Any bicellular map $\mathfrak{b}_{g}$ together with one 
of its $2(g+1)$ trisections is mapped via $\Xi_b$ into
a bicellular map of lower genus or a pair of unicellular 
maps. Note that either the topological genus decreases by at least one
or we lose connectivity and decompose into a pair of unicellular
maps. From \citep{Chapuy:11}, we know that a unicellular
map can be iteratively sliced into a planar tree.
Therefore, we have
%%%
%%%%%%%%%%%%%%%%%%%%%%%%%%%%%%%%%%%%%%%%%%%%%%%%%%%%%%%%%%%%%%%%%%%%%%%%%%%%%%%
%%%
\begin{corollary}\label{C:iterate}
Any bicellular map can be sliced into a pair of planar trees
and we have
\begin{equation} \label{E:induction2}
\begin{split}
&  {\sf B}_g(n)  =
\sum\limits_{g_0 < g_1< g_i<\cdots <g_{r-1} <g_r=g} 
 \sum_{b=1}^{r} \sum_{m=0}^{n} \\
&\left(  \sum_{l \geq 0}^{b-1} {b-1 \choose l}
 \prod_{i=i_v \in I_l} \frac{1}{2g_{i,A}} {m+1-2 g_{(i_{v-1},A)}
  \choose 2(g_i -g_{i-1}) +1} \right.\times \\ 
& \prod_{i=j_v \in J_{b-l-1}}
 \frac{1}{2g_{i,B}} { n-m+1-2g_{(j_{v-1},B)} \choose 2(g_i-g_{i-1})+1}\times  \\
 &\left. \left( \sum_{k \geq 1}^{2(g_b - g_{b-1})+1}  \frac{1} {2g_b+2} {m+1-2(g_{b-1,A})\choose k } 
 { n-m+1-2(g_{b-1,B}) \choose 2(g_b-g_{b-1})+1-k}\right) \right) \\
 & \times \prod_{i=b+1}^{r} \frac{1}{2g_i+2}{n-2g_{i-1} 
 \choose 2(g_i-g_{i-1})+1} \epsilon_0(m)\epsilon_0(n-m),
\end{split}
\end{equation}
where 
\begin{eqnarray*}
I_l&=&\{i_1,i_2,\cdots i_l\}, \ \ 
J_{b-l-1}=\{j_1,j_2,\cdots j_{b-l-1}\},\\
g_{b-1,A}&=&\sum\limits_{i_x=i_1}^{i_l}(g_{i_x}-g_{i_x-1}), \ \ 
g_{b-1,B}= g_{b-1}-g_{b-1,A},\\
g_{(i_{v-1},A)}&=&\sum\limits_{i_x=i_1}^{i_{(v-1)}}(g_{i_x}-g_{i_x-1}), \ 
 g_{(j_{v-1},B)}=\sum\limits_{j_y=j_1}^{j_{(v-1)}}(g_{j_y}-g_{j_y-1}),
\end{eqnarray*}
and $\epsilon_0(n)$ denotes the number of plane trees with $n$ edges.
\end{corollary}

%%%
%%%%%%%%%%%%%%%%%%%%%%%%%%%%%%%%%%%%%%%%%%%%%%%%%%%%%%%%%%%%%%%%
%%%
\section{Uniform generation}\label{S:uni}
%%%
%%%%%%%%%%%%%%%%%%%%%%%%%%%%%%%%%%%%%%%%%%%%%%%%%%%%%%%%%%%%%%%%
%%%
In this section, we show how to generate a bicellular map of 
given genus $g$ over $n$ edges with uniform probability. 

Here is the key idea: according to Corollary~\ref{C:iterate}, 
any bicellular map  decomposes into to a pair of plane trees 
$(\mathfrak{u}_0(m), \mathfrak{u}_0(n-m))$. Recruiting the inverse to 
slicing, the gluing, we can recover $\mathfrak{b}_g$. As each
bicellular map is generated with multiplicity $2(g+1)$, see 
Corollary~\ref{T:Rec-Bi}, we can employ our bijection to uniformly
generate bicellular maps of fixed topological genus $g$.

We first give the definition of glue path.

To this end, let $\mathfrak{b}_{g}$ denote a bicellular map of genus 
$g$ having $n$ edges, let $\mathfrak{p}_g$ denote a pair of
unicellular maps of genus sum $g$ and let $\mathfrak{m}$ denote
a map.
%%%%
%%%%%%%%%%%%%%%%%%%%%%%%%%%%%%%%%%%%%%%%%%%%%%%%%%%%%
%%%
\begin{definition}
A glue path starting from $\mathfrak{p}_0$ to $\mathfrak{b}_g$,
is  a sequence 
$$ 
\left( (\mathfrak{m}^0=\mathfrak{p}_{g_0=0},j_0=0),\ldots, 
(\mathfrak{m}^i,j_i), \ldots, 
(\mathfrak{m}^b=\mathfrak{b}_{g_b},j_b=1), \ldots, 
(\mathfrak{m}^{r}=\mathfrak{b}_{g},j_r=1\right),
$$
where $j_i\in \{0,1\}$ is a flag, an indicator variable for 
connectivity. $b$ is the first step where $j_i$ switches to $1$.
The corresponding sequence 
$$
\left( (g_0=0,j_0=0),\ldots, 
(g_i,j_i), \ldots,(g_r,j_r=1)\right).
$$ 
is called the signature of the glue path.
\end{definition}
%%%
%%%%%%%%%%%%%%%%%%%%%%%%%%%%%%%%%%%%%%%%%%%%%%%%%%%%%%%%%%%%%%%%%%
%%%

%%%
%%%%%%%%%%%%%%%%%%%%%%%%%%%%%%%%%%%%%%%%%%%%%%%%%%%%%%%%%%%%%%%%%%%
%%%

We shall generate $\mathfrak{b}_g(n)$ in two steps: first we construct 
a pair of planar trees $\mathfrak{p}_0=(\mathfrak{u}_0(m),\mathfrak{u}_0(n-m))$ 
with $n$ edges with uniform probability. There are  $\sum_{m=0}^{n} \epsilon_0(m)
\epsilon_0(n-m)$ such pairs. Second, starting from this pair, we generate a 
glue path to the target genus.

It is well-known how to generate a plane tree with $n$ 
edges in linear time \citep{Duchon}. For every pair
$\mathfrak{p}_0=(\mathfrak{u}_0(m),\mathfrak{u}_0(n-m))$, 
we next generate a glue path with uniform probability as follows.

For a given pair of unicellular maps $\mathfrak{p}_0$ and target genus $g$,
we first construct all signatures. 

For every such path we have $(g_0,j_0)=(0,0)$ and $(g_r,j_r)=(g,1)$.  
We can construct the signatures inductively. The induction basis is trivial,
as for the step, suppose we have arrived at $(g_i,j_i)$, where $0\leq g_i \leq 
g$ and $j_i=0$ or $1$. If $j_i=0$, then we can generate either 
$\{(g_i+1,0),(g_i+2,0), \ldots (g,0)\}$ or $\{(g_i,1),(g_i+2,1), \ldots (g,1)\}$.
If $j_i=1$, then we obtain $\{(g_i+1,1),(g_i+2,1), \ldots (g,1)\}$. 
Since the initial and final tuples are fixed, after finitely many iterations
we can thereby generate all the signatures, see Fig.~\ref{F:path}.
%%%
%%%%%%%%%%%%%%%%%%%%%%%%%%%%%%%%%%%%%%%%%%%%%%%%%%%%%%%%%
%%%
\begin{figure}[ht]
\begin{center}
\includegraphics[width=0.8\textwidth,height=0.4\textwidth]{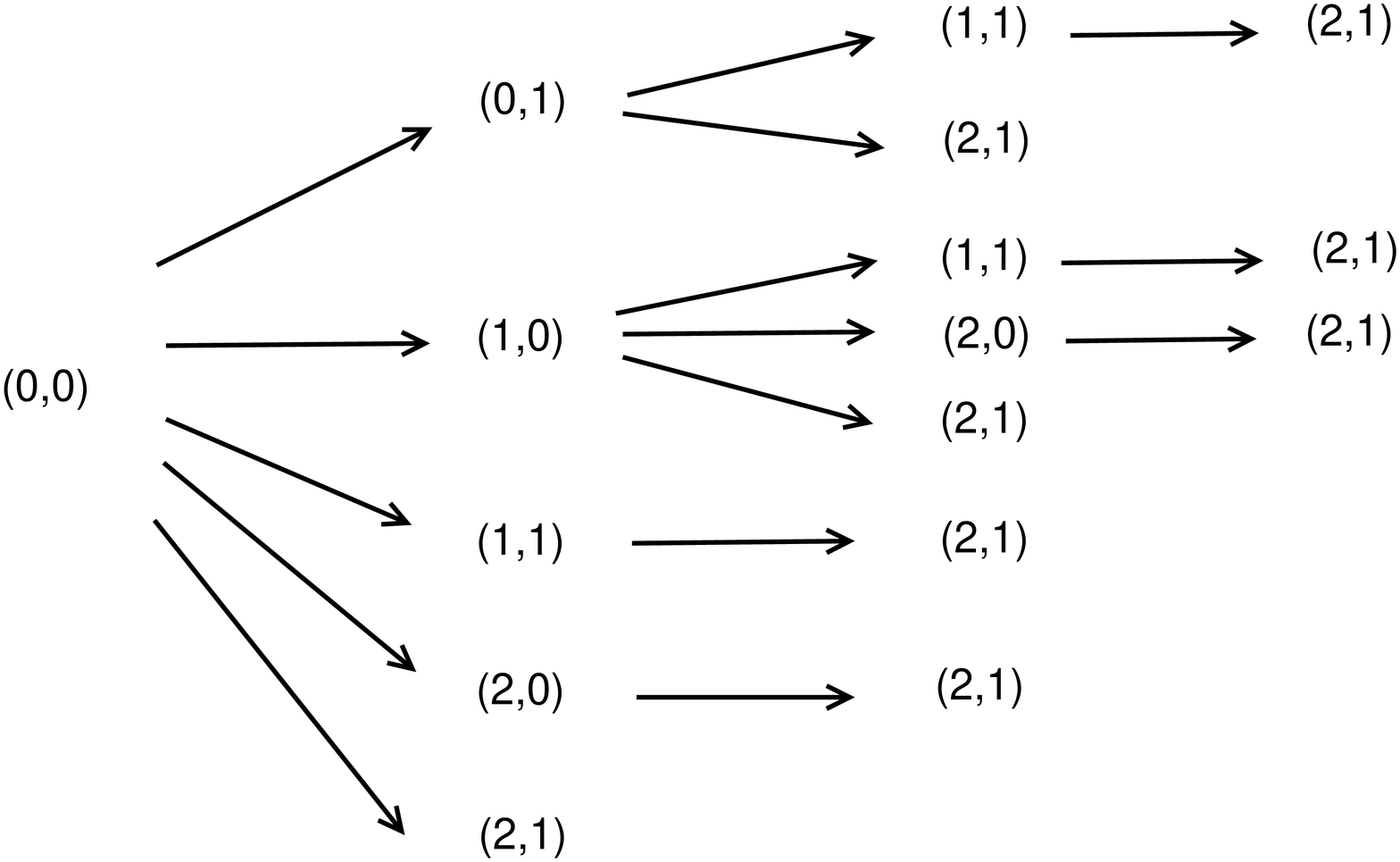}
\end{center}
\caption{the signatures for target genus $2$.}
\label{F:path}
\end{figure}
%%%
%%%%%%%%%%%%%%%%%%%%%%%%%%%%%%%%%%%%%%%%%%%%%%%%%%%%%%%
%%%

Every signature path has a probability. It is given by
the number of glue paths from $(\mathfrak{p}_0,0)$ to 
$(\mathfrak{m}_{g_{i+1}}, j_{i+1})$ having this 
signature, normalized by the total number of glue 
paths.
  
We arrive at
\begin{equation}\label{signa}
\begin{split}
\mathbb{P}&((g_{i+1}=t ,j_{i+1})\mid  (g_0,0) \ldots, (g_i,j_i),(g,1)) \\
&=\frac{\sum\limits_{(t_0=g_0,0) \ldots, 
(t_i=g_i,j_i), 
(t_{i+1}=g_{i+1}, j_{i+1})\cdots (t_r=g,1) }
\varOmega_{i+1}}
{\sum\limits_{(t_0=g_0,0) \ldots, (t_i=g_i,j_i) \cdots
(t_r=g,1) }\varOmega_{i}},
\end{split}
\end{equation}
where the $\varOmega_{i}$ denotes the sum of weights over all
the signatures that contain $(g_i,j_i)$.
We can thus compute all transition probabilities of
sates of signatures state and derive the transition matrix $M$.

Given $M$, we can construct a glue path as follows:

Suppose we are at step $i$ and we have constructed a map
$\mathfrak{m}^i$ with $(g_i,j_i)$. Then $(g_{i+1},j_{i+1})$ 
can be derived via $M$, by the process ${\tt NextTuple}$.

Next we select the vertices and gluing accordingly. Let $(A,B)$ be
the pair of plane trees. Suppose we have the tuple $(m_{g_i}, j_{i})$, 
${\tt NextTuple}$ produces $(g_{i+1}, j_{i+1})$ and if $j_{i+1}=0$, 
then we select the vertices ``locally'' on one of the unicellular 
maps via 
$$
\mathbb{P}(V_{i,1})= \frac{1}{{m+1-2g_{i,A}\choose 2(g_{i+1}-g_i)+1}}
$$
or
$$
\mathbb{P}(V_{i,2})= \frac{1}{{n-m+1-2g_{i,B}\choose 2(g_{i+1}-g_i)+1}}.
$$
In case of $j_{i+1}=1$ and $j_{i}=0$, we need to choose the vertices 
distributed i.e.~
$$
\mathbb{P}(V_{i,3})= \frac{1}{\sum_{k=1}\limits^{2(g_{i+1}-g_{i})}
{m+1-2g_{i,A}\choose k}
{n-m+1-2(g_{i}-g_{i,A}) \choose 2(g_{i+1}-g_{i})+1-k}}.
$$
Finally in case of $j_{i+1}=1$ and $j_{i}=1$ we can choose vertices 
arbitrarily, i.e.~
$$
\mathbb{P}(V_{i,4})= 
\frac{1}{{n-2g_{i}\choose 2(g_{i+1}-g_{i})+1}}.
$$

We refer to the above three cases of local, distributed and free
vertex selection as {\tt SelectVertex1}, {\tt SelectVertex2} and 
{\tt SelectVertex3}.

After the sequence of vertices $V_i$ is selected, a bicellular map 
$\mathfrak{b}^{i+1}$ is constructed by the process {\tt Glue}. Notice that in
accordance with Theorem~\ref{T:main} after every application of
{\tt Glue}, we normalize by $2g_i$, for $j_i=0$, or $2g_i+2$ in case 
of $j_i=1$. We present the pseudocode of the procedures in 
Algorithm~\ref{A:bi}.

%%%
%%%%%%%%%%%%%%%%%%%%%%%%%%%%%%%%%%%%%%%%%%%%%%%%%%%%%%%%%%%%%%%%%%%%%%%%%%
%%%
\begin{algorithm}
\begin{algorithmic}[1]
\STATE {\tt UniformBi-Matching}~($\mathfrak{m}^0=\mathfrak{p}^0, Targettuple$)
\STATE {$i \leftarrow 0$,$j \leftarrow 0$}
\WHILE {$(g_i,j_i) \neq Targettuple$}
\STATE {$(g_{i+1},j_{i+1}) \leftarrow {\tt NextTuple}~((g_0,0) \ldots, (g_i,j_i), Targettuple)$}
\IF    {$j_{i+1}=0$}
\STATE {$V_i \leftarrow {\tt SelectVertex1}~(\mathfrak{m}^i, 2(g_{i+1}-g_i)+1)$}
\ELSIF {$j_{i+1}=1 \& j_{i} =0 $}
\STATE {$V_i \leftarrow {\tt SelectVertex2}~(\mathfrak{m}^i, 2(g_{i+1}-g_i)+1)$}
\ELSIF {$j_{i+1}=1 \& j_{i} =1 $}
\STATE {$V_i \leftarrow {\tt SelectVertex3}~(\mathfrak{m}^i, 2(g_{i+1}-g_i)+1)$}
\ENDIF
\STATE {$\mathfrak{m}^{i+1} \leftarrow {\tt Glue}~(\mathfrak{m}^i, V_i)$, $i\leftarrow i+1$}
\ENDWHILE
\STATE \textbf{return} $\mathfrak{m}^i$
\caption {\small Generation of glue path for bicellular maps }
\label{A:bi}
\end{algorithmic}
\end{algorithm}
%%%
%%%%%%%%%%%%%%%%%%%%%%%%%%%%%%%%%%%%%%%%%%%%%%%%%%%%%%%%%%%%%%%%%%%%%%%%%
%%%

Since the target genus is fixed constant and since the intermediate genera
are monotone, the while-loop of Algorithm 1 is executed only a constant 
number of times. Using appropriate memorization techniques, {\tt NextTuple} 
and {\tt Glue} can be implemented in constant time. Furthermore 
{\tt SelectVertex1}, {\tt SelectVertex2} and {\tt SelectVertex3} have linear 
run-time complexity. Thus, combined with a linear time sampler for planar trees, 
our approach allows for the uniform generation of random $2$-backbone matchings 
in $O(n)$ time.

We accordingly obtain
%%%
%%%%%%%%%%%%%%%%%%%%%%%%%%%%%%%%%%%%%%%%%%%%%%%%%%%%%%%%%%%%%%%%
%%%
\begin{corollary}
Algorithm \ref{A:bi} generates uniformly bicellular maps.
\end{corollary}
%%%
%%%%%%%%%%%%%%%%%%%%%%%%%%%%%%%%%%%%%%%%%%%%%%%%%%%%%%%%%%%%%%%%
%%%
\begin{proof}
First, we generate the plane trees uniformly. Second, by construction, every  
transition from $\mathfrak{m}^i$ to $\mathfrak{m}^{i+1}$ is a bijection by 
Theorem~\ref{T:main} and uniform after normalizing by $2(g_{i+1}+1)$, whence the
corollary.
\end{proof}

Now we can extend our result in order to generate $2$-backbone diagrams of 
genus $g$ with uniform probability. The idea is as follows: first we 
uniformly generate a $2$-backbone matching of genus $g$ with $n$ arcs. 
Then we choose $(\ell-2n)$ unpaired vertices and insert them into the 
matching.

Let $\mathbb{P}_{d}(t=n|\ell, g)$ denote the probability of the 
$2$-backbone diagram of length $\ell$ and genus $g$ having exactly 
$n$ arcs, $0\le n \le \lfloor \ell/2\rfloor$. 
Let $\delta_g(\ell)$ denote the number of $2$-backbone diagrams of genus 
$g$ over $\ell$ vertices. Furthermore, let $\delta_g(\ell ,n)$ denote the 
number of diagrams of genus $g$ over $\ell$ vertices having exactly $n$ 
arcs, $2n\le \ell$. Denote the number of $2$-backbone matchings with 
$n$ arcs by $\xi_g(n)$. Then
\begin{eqnarray*}
\delta_g(\ell, n) & = &  {\ell \choose \ell-2n} \xi_g(n)\\
\delta_g(\ell) & = & \sum_{n=0}^{\lfloor l/2\rfloor} \delta_g(\ell, n)
=\sum_{n=0}^{\lfloor l/2\rfloor} {\ell \choose \ell-2n} \xi_g(n)
\end{eqnarray*}
and $\mathbb{P}_{d}(t=n|\ell, g)=\delta_g(\ell, n)/\delta_g(\ell)$. 

This leads to Algorithm~\ref{A:diagram}, which generates uniformly 
diagrams of length $\ell$ of genus $g$. Accordingly, a $2$-backbone 
diagram of genus $g$ over $\ell$ vertices
with exactly $n$ arcs is generated, which we denote by $D_g^2(l)$. 

%%%
%%%%%%%%%%%%%%%%%%%%%%%%%%%%%%%%%%%%%%%%%%%%%%%%%%%%%%%%%%%%%%%%%%%%%%%%%%
%%%
\begin{algorithm}
\begin{algorithmic}[1]
\STATE {\tt Uniform2BackboneDiagram}~($\ell, TargetGenus$)
\STATE {$n\leftarrow {\tt NumberofArcs}(\ell, g)$}
\STATE {$\mathfrak{p}_0 \leftarrow {\tt UnifomTrees}(n)$}
\STATE {$\mathfrak{b}_g \leftarrow {\tt UniformBi-Matching}(\mathfrak{p}_0, TargetGenus)$}
\STATE {$ D_g^2(l) \leftarrow {\tt InsertUnpairedVertices}( \mathfrak{b}_g , \ell)$}
\STATE \textbf{return} $D_g^2(l)$
\caption {\small }
\label{A:diagram}
\end{algorithmic}
\end{algorithm}
%%%
%%%%%%%%%%%%%%%%%%%%%%%%%%%%%%%%%%%%%%%%%%%%%%%%%%%%%%%%%%%%%%%%%%%%%%%%%
%%%
As for the subroutines:
\begin{itemize}
\item  {\tt NumberofArcs} returns $n$ with probability 
       $\mathbb{P}_{d}(t=n|\ell, g)$ and accordingly gives the number 
       of arcs in $2$-backbone diagram of length $\ell$,
\item {\tt UniformTrees} uniformly generates a pair of 
           planer trees with a total of $n$ arcs,
\item {\tt UniformBi-Matching} generates a $2$-backbone matching of genus 
       $g$ with $n$ arcs,
\item {\tt InsertUnpairedVertices} selects $(\ell-2n)$ vertices from 
      $\ell$ vertices as to be unpaired and inserts them.
\end{itemize}

The result of some experiments conducted in connection with the
generation of random matchings and diagrams are displayed
in Fig.~\ref{F:unifor}.
%%%
%%%
%%%%%%%%%%%%%%%%%%%%%%%%%%%%%%%%%%%%%%%%%%%%%%%%%%%%%%%%%
%%%
\begin{figure}[ht]
\begin{center}
\includegraphics[width=0.4\textwidth,height=0.4\textwidth]{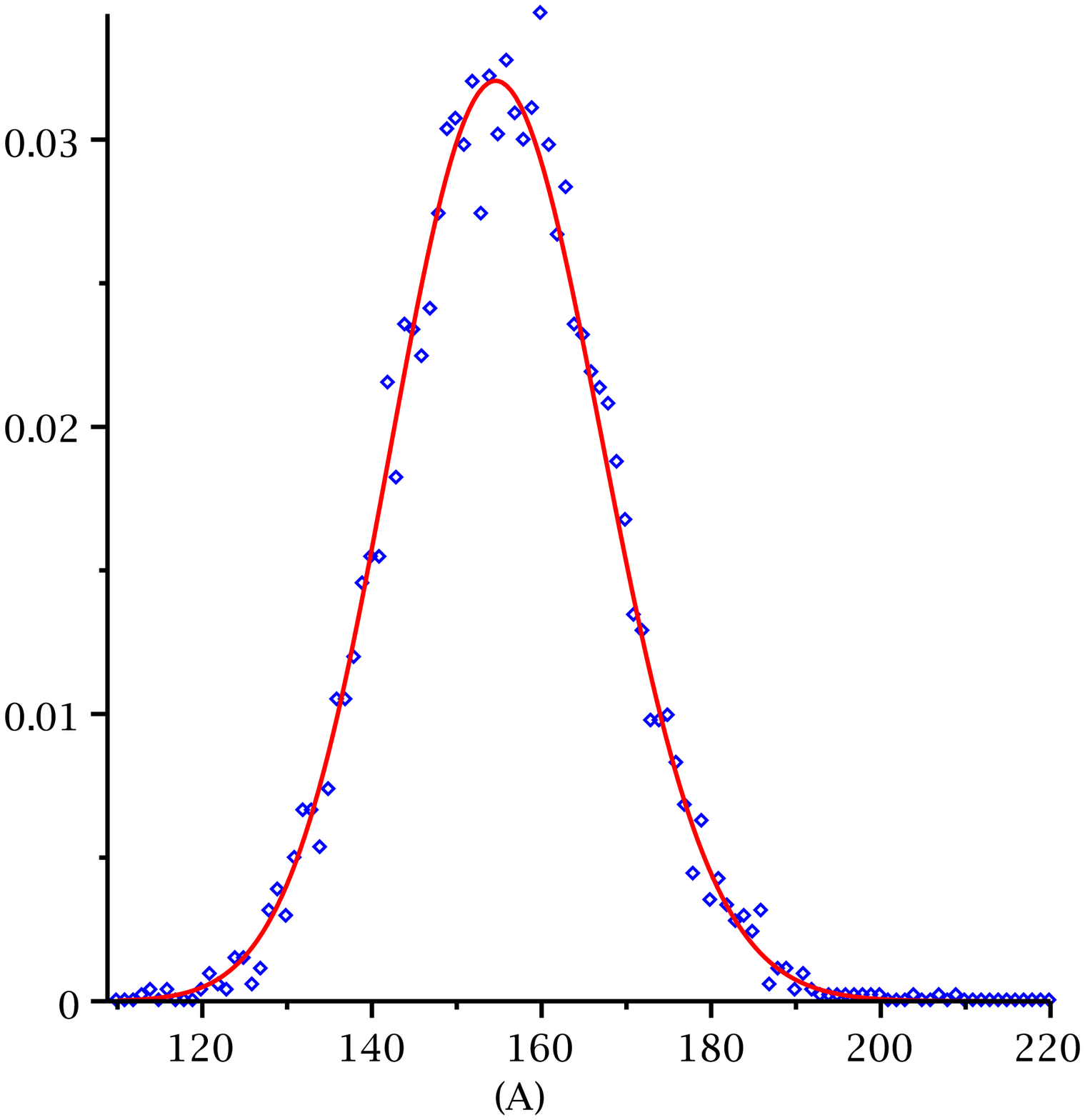}
\includegraphics[width=0.4\textwidth,height=0.4\textwidth]{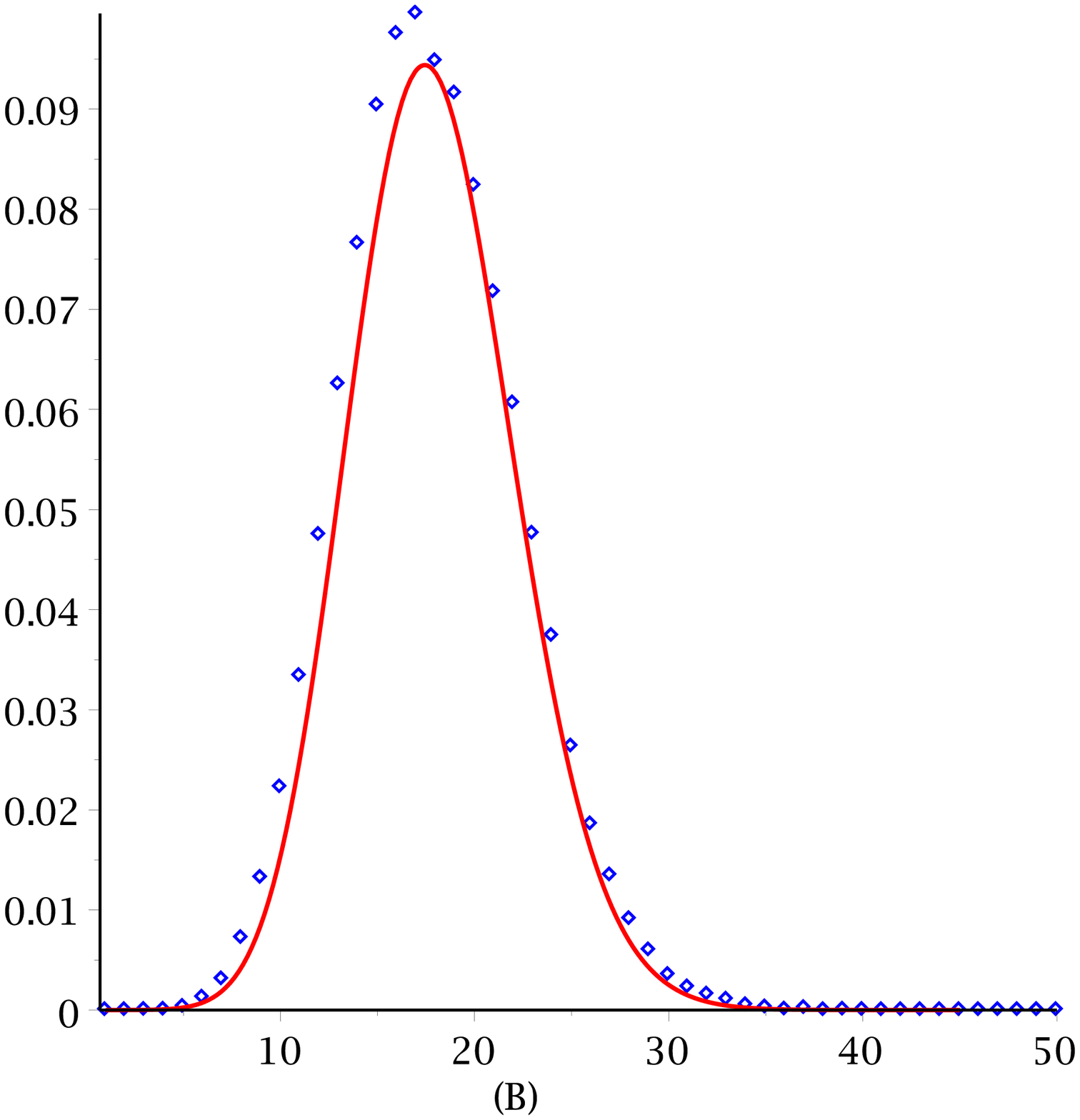}
\end{center}
\caption{ \small Uniform generation: $(A)$ matchings over $2$ backbones, $n=10$, $g=1$. 
 We generate $N=10^6$ matchings and display the frequencies of their multiplicities
 (blue dots) together with the Binomial coefficient of the uniform sampling.
 (B) The analog of $(A)$ for diagrams, i.e.~incorporating isolated vertices. Here we have $n=10$, $g=1$,  We generate $N=10^6$  diagrams and display the frequencies of their multiplicities (blue dots) together with  the Binomial coefficient of the uniform sampling.}
\label{F:unifor}
\end{figure}
%%%
%%%%%%%%%%%%%%%%%%%%%%%%%%%%%%%%%%%%%%%%%%%%%%%%%%%%%%%
%%%
%%%%%%%%%%%%%%%%%%%%%%%%%%%%%%%%%%%%%%%%%%%%%%%%%%%%%%%%%%%%%%%%
%%%

\section{Discussion}\label{S:dis}

%%%
%%%%%%%%%%%%%%%%%%%%%%%%%%%%%%%%%%%%%%%%%%%%%%%%%%%%%%%%%%%%%%%%
%%%

%%%
%%%%%%%%%%%%%%%%%%%%%%%%%%%%%%%%%%%%%%%%%%%%%%%%%%%%%%%%%%%%%%%%
%%%

We derived a uniform generation algorithm for RNA-RNA interaction structures of fixed
topological genus. The algorithm is very fast having only linear time complexity.
It allows immediately to obtain an abundance of statistical data on these structures.

In the following we shall consider biased sampling of RNA-RNA interaction structures. 
The bias is obtained by employing a simplified version of extending the bio-physical
loop-energy model of RNA secondary structures to RNA-RNA interaction structures. 
Here we restrict ourselves to the case of genus $0$ structures, but the treatment
of higher genera structures is straightforward from here. Note that genus $0$ 
interaction structures exhibit in general cross serial interactions between their two
backbone, i.e.~exhibit crossing arcs.

RNA structures are, due to the biophysical context, subject to specific constraints 
with respect to their free energy \cite{zuker1999}. The latter is oftentimes modelled 
as a function of the loops of the underlying RNA structure \cite{zuker1999}. 
This goes back to Waterman {\it et al.} 
\cite{waterman1978, waterman1978secondary, zuker1984rna, nussin, kleitman1970} 
who realized that the classic secondary structure recursion is compatible with the 
loop energy model. It is interesting to note that these loops actually correspond 
to faces in the fat graph model, that is boundary components. This phenomenon 
naturally extends to structures over any number of backbones and any topological 
genus. Their loops are also just topological boundary components and the framework
extend in a natural way, see Fig.~\ref{F:bb}. In case of RNA 
secondary structures, we find essentially three types of loops: hairpin loops,
interior loops (including helices and bulge loops) and multi-loops. 
The Poincar\'{e} duality described in Lemma~\ref{L:diagram} interchanges boundary
components and vertices, whence we have the following correspondences
\begin{itemize}
\item hairpin loops and vertices of degree one, 
\item interior loops and vertices of degree two and 
\item multi-loops and vertices of degree greater than two, see Fig. \ref{F:loop3}.
\end{itemize}
 %%%
 %%%%%%%%%%%%%%%%%%%%%%%%%%%%%%%%%%%%%%%%%%%%%%%%%%%%%%%%%%
 %%%
\begin{figure}[H]
\begin{center}
\includegraphics[width=0.8\textwidth]{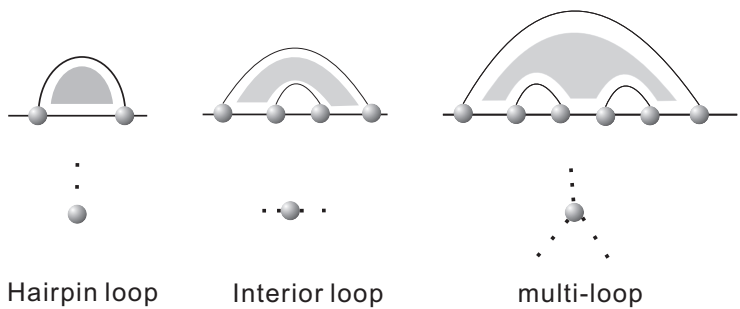}
\end{center}
\caption{The three loop types in the RNA secondary structures.}
\label{F:loop3}
\end{figure}
%%%
%%%%%%%%%%%%%%%%%%%%%%%%%%%%%%%%%%%%%%%%%%%%%%%%%%%%%%%%%%%%
%%%

Let $B_g$ denote an RNA-RNA interaction structure having length
$l$, $n$ arcs and genus $g$. Lemma~\ref{L:diagram} associates to
this diagram and a bicellular map, denoted by $\mathfrak{b}_g$.

In Section~\ref{S:genus} we discussed that given a bicellular map
$\mathfrak{b}_g$ together with a distinguished trisection, a finite 
number of vertex-slicings produces a pair of plane trees together with 
a collection of labeled vertices.

We showed in Theorem~\ref{T:main} that any such slicing is reversible, whence 
the decomposition via vertex slicings is unique. This means that we actually
have derived an unambiguous grammar that decomposes any RNA-RNA interaction 
structure of fixed topological genus into an (ordered) pair of secondary 
structures with some labeled loops.

Then the energy of such a structure$\eta(B_g)$ is given by
$$
\eta(B_g)=n\cdot b+ \eta(S_1^p)+\eta (S_2^{q}) + L_g,
$$
where $b$ is the energy contribution of an arc. $S_1^p$ ($S_2^q$) is a 
secondary structure with $p$ ($q$) marked loops. Furthermore, $\eta(S_1^p)=
\sum_{X_1} L^{X_1}$ and  $\eta(S_2^q)=\sum_{X_2} L^{X_2}$, where $X_1$ ($X_2$) 
are the set of unmarked loops in $S_1$ ($S_2$). 
$L_g$  represents the energy contribution of the labeled loops and the 
contribution of the gluing path.

To illustrate what happens here, let us have a look at the case of $g=0$. 
Suppose we are given a genus $0$ matching over $2$ backbones, $B_0$. According 
to Theorem~\ref{T:main} its dual bicellular map, $\mathfrak{b}_0$, corresponds to a 
pair of trees and together with three labeled vertices, denoted by $(T_1, T_2)$. 
The corresponding two  pseudoknot-free secondary structure with three labeled 
boundary components are denoted by $(S_1\cup S_2)$. Thus we have 
$\eta(B_0)=\eta(S_1^{1})+\eta (S_2^{2}) + L_0=\eta (S_1^1)+\eta(S_2^2)+ L^{mul} + 
\epsilon$, where $L_0=L^{mul} + \epsilon$, since 
the three vertices after gluing will form a vertex of degree at least $3$, 
corresponding to a multi-loop. Finally, $\epsilon$ can be regarded as the 
contribution of the particular type of pseudoknot being glued.
The situation is particularly transparent for $g=0$, since there are only two 
shapes $E$ and $F$, where a shape is a diagram without unpaired vertices and
$1$-arcs in which all stacks (parallel arcs) have size one.
These two shapes are depicted in Fig.~\ref{F:genus01} and Fig.~\ref{F:genus02}, 
where we show in addition these two shapes and the pair of secondary structures 
with three labeled boundary components they slice into.
we accordingly derive  
$$
\eta(E)= \eta(S_1^1) + \eta(S_2^2)+ L_0=L^{mul} + \epsilon, \quad
\eta(F)= \eta(S_1^1) + \eta(S_2^2)+ L_0= 2L^{mul} + \epsilon. 
$$
%%%
%%%%%%%%%%%%%%%%%%%%%%%%%%%%%%%%%%%%%%%%%%%%%%%%%%%%%%%%%
%%%
\begin{figure}[H]
\begin{center}
\includegraphics[width=0.8\textwidth]{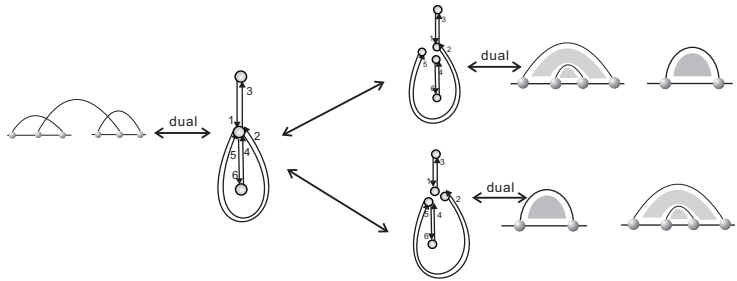}
\end{center}
\caption{ \small the $E$-shape and the pair of secondary 
          structures it slices into.}
\label{F:genus01}
\end{figure}
%%%
%%%%%%%%%%%%%%%%%%%%%%%%%%%%%%%%%%%%%%%%%%%%%%%%%%%%%%%
%%%
%%%
%%%%%%%%%%%%%%%%%%%%%%%%%%%%%%%%%%%%%%%%%%%%%%%%%%%%%%%%%
%%%
\begin{figure}[H]
\begin{center}
\includegraphics[width=0.8\textwidth]{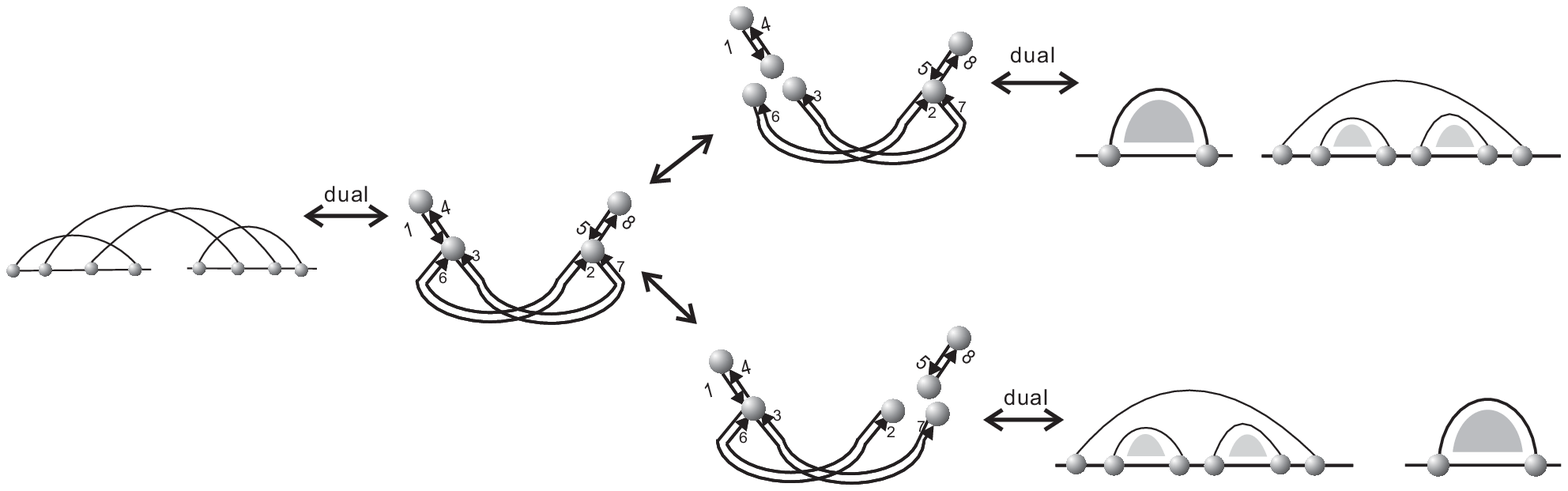}
\end{center}
\caption{ \small the $F$-shape and the pair of secondary 
          structures it slices into. }
\label{F:genus02}
\end{figure}
%%%
%%%%%%%%%%%%%%%%%%%%%%%%%%%%%%%%%%%%%%%%%%%%%%%%%%%%%%%
%%%

%%%
%%%%%%%%%%%%%%%%%%%%%%%%%%%%%%%%%%%%%%%%%%%%%%%%%%%%%%%%%
%%%
\begin{figure}[H]
\begin{center}
\includegraphics[width=0.8\textwidth]{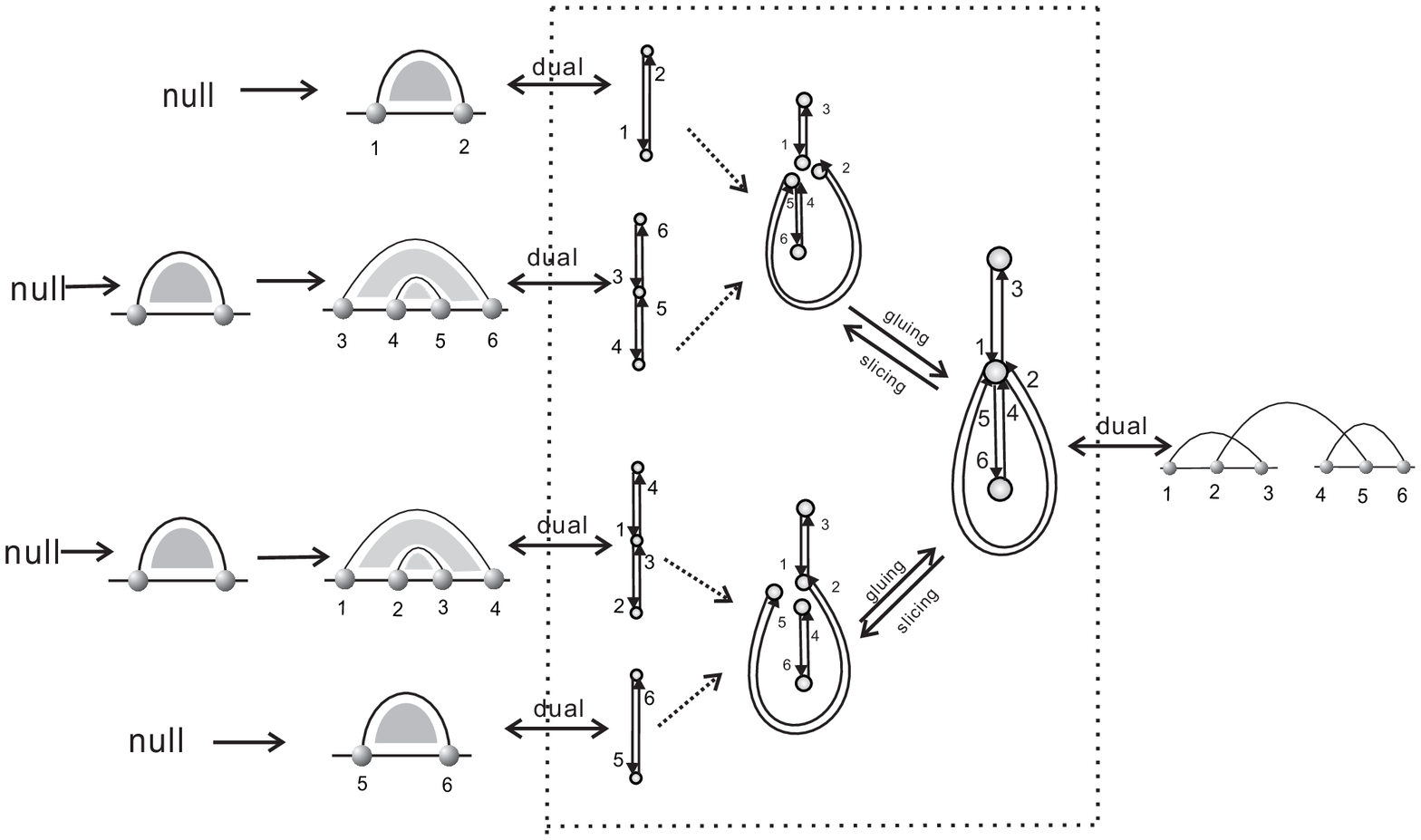}
\end{center}
\caption{ \small Howto generate the $E$-shape. left: 
generating the two secondary structures, middle: 
gluing the corresponding dual maps of the latter two, 
right: $E$-shape obtained by dualizing again. }
\label{F:genus001}
\end{figure}
%%%
%%%%%%%%%%%%%%%%%%%%%%%%%%%%%%%%%%%%%%%%%%%%%%%%%%%%%%%
%%%

It is important to note that our sampler is based on a two 
literally ``orthogonal'' compositions.  
The first is inductive in length and adds either unpaired vertices 
or arcs. There is no topological ``complexity'' in the structures. 
Point in case: any RNA secondary structures has genus $0$.
The second is inductive in either topological genus or connectedness
but adds neither vertices nor arcs. This induction is novel and 
substantially different from length based induction, See Fig. \ref{F:genus001}
for an illustration for the generation of $E$-shape of genus $0$. 

We shall proceed and study several statistics of loops in RNA-RNA interaction
structures, see Fig.~\ref{F:bb}.
%%%
%%%%%%%%%%%%%%%%%%%%%%%%%%%%%%%%%%%%%%%%%%%%%%%%%%%%%%%%%
%%%
\begin{figure}[H]
\begin{center}
\includegraphics[width=0.8\textwidth,height=0.4\textwidth]{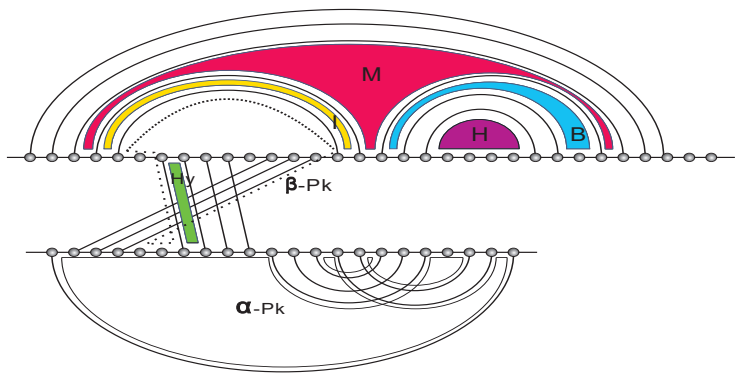}
\end{center}
\caption{ \small RNA-RNA interaction structure and its loops: multi-loop, $M$,
 interior loop, $I$, hairpin loop, $H$, bulge loop, $B$, exterior stack
 $H_y$,  pseudoknot loop over one backbone, $\alpha$-$Pk$, and 
 pseudoknot loop over two backbones $\beta$-$Pk$. }
\label{F:bb}
\end{figure}
%%%
%%%%%%%%%%%%%%%%%%%%%%%%%%%%%%%%%%%%%%%%%%%%%%%%%%%%%%%
%%%

We first present the distribution of loop types in interaction structures of 
genus $g$, see Fig.~\ref{F:bb}. 
We shall distinguish loops that contain only edges with endpoints on one backbone
($\alpha$-loops) and those that contain also edges connecting the two backbones 
($\beta$ loops), see Fig.~\ref{F:loop1} and \ref{F:loop2}.

%%%
%%%%%%%%%%%%%%%%%%%%%%%%%%%%%%%%%%%%%%%%%%%%%%%%%%%%%%%%%
%%%
\begin{figure}[H]
\begin{center}
\includegraphics[width=0.45\textwidth,height=0.45\textwidth]{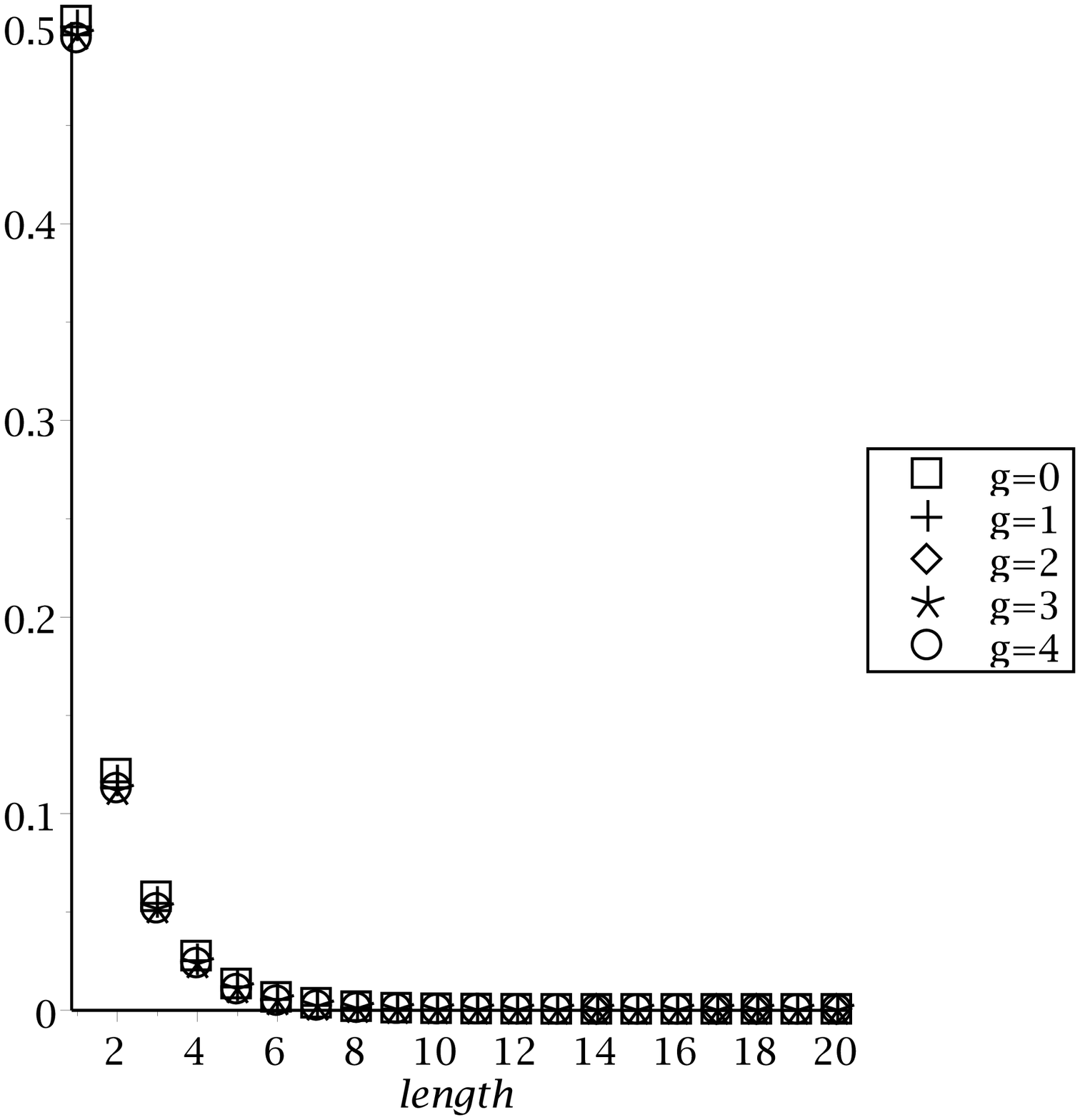}
\includegraphics[width=0.45\textwidth]{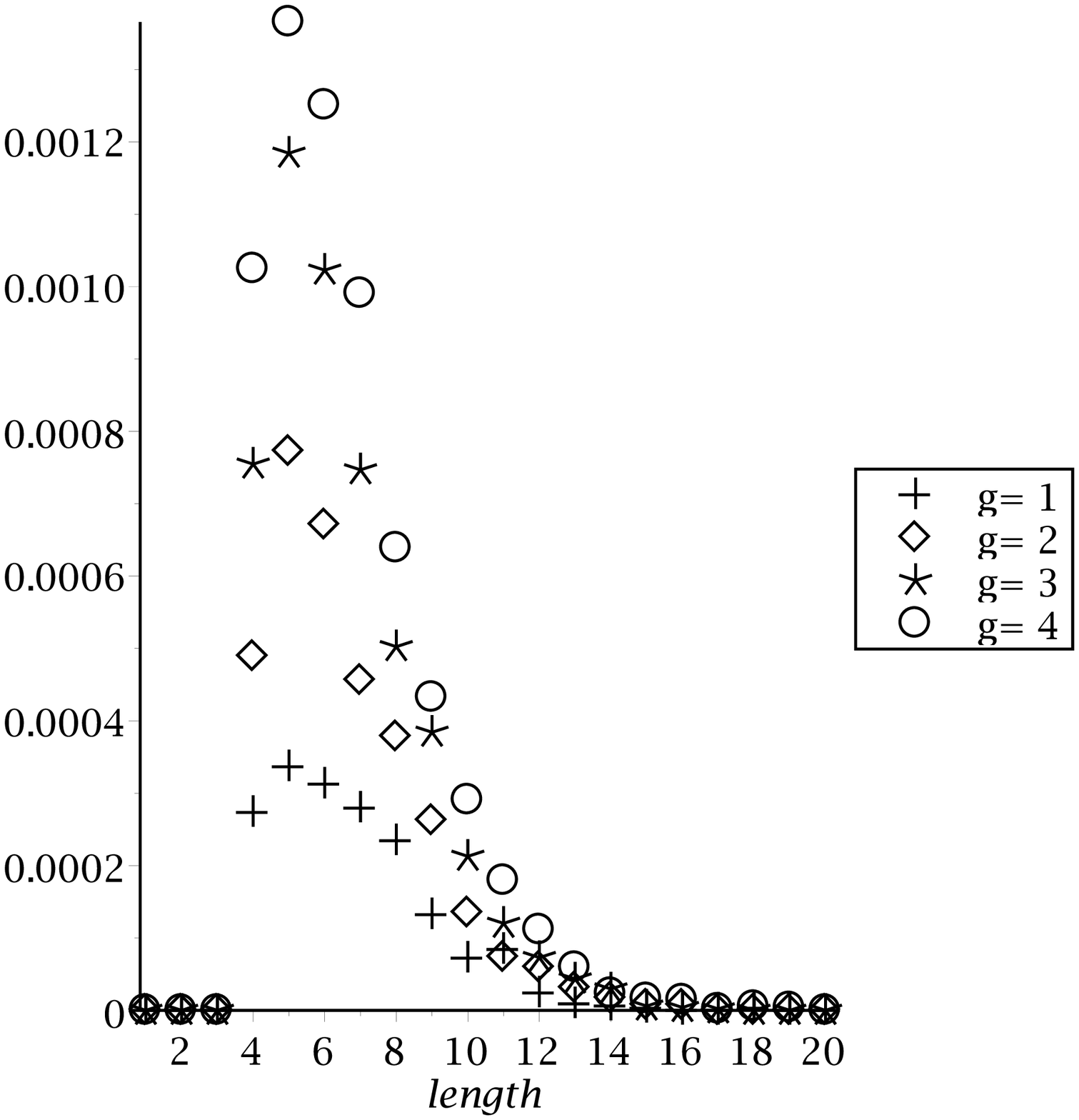}
\end{center}
\caption{ \small  The length of $\alpha$-loops in uniformly 
 generated, genus filtered, RNA structures. 
 Data are obtained from $10^5$ interaction structures of length $500$
 with genus ranging from $0$ to $4$, respectively. 
 left: distribution of standard $\alpha$-loops, where the $x$-axis represents 
 the length of the boundary component and $y$-axis denotes frequency. 
 right: distribution of pseudoknot $\alpha$-loops.}
\label{F:loop1}
\end{figure}
%%%
%%%%%%%%%%%%%%%%%%%%%%%%%%%%%%%%%%%%%%%%%%%%%%%%%%%%%%%
%%%

%%%%%%%%%%%%%%%%%%%%%%%%%%%%%%%%%%%%%%%%%%%%%%%%%%%%%%%%%
%%%
\begin{figure}[H]
\begin{center}
\includegraphics[width=0.45\textwidth,height=0.45\textwidth]{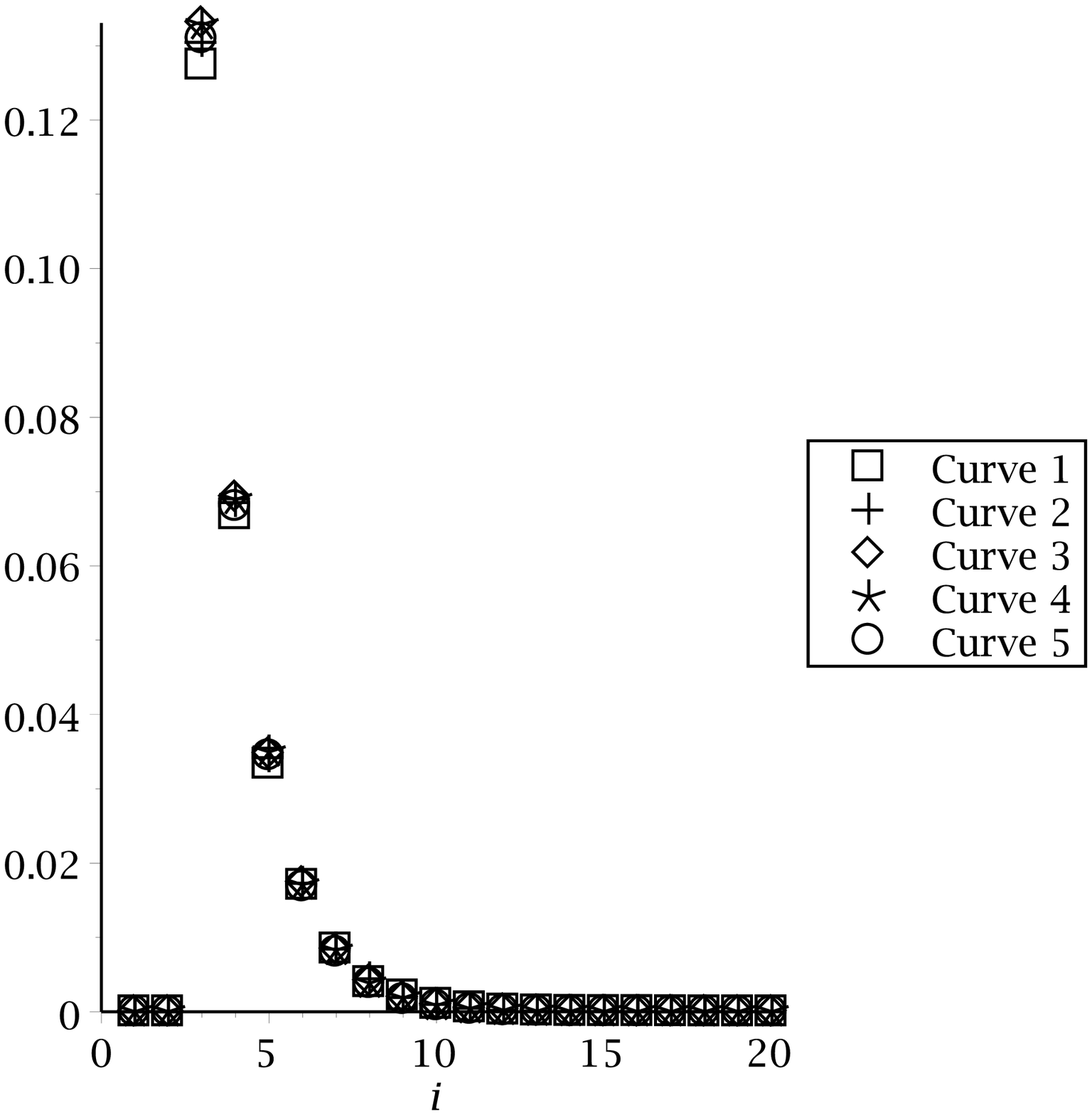}
\includegraphics[width=0.45\textwidth,height=0.45\textwidth]{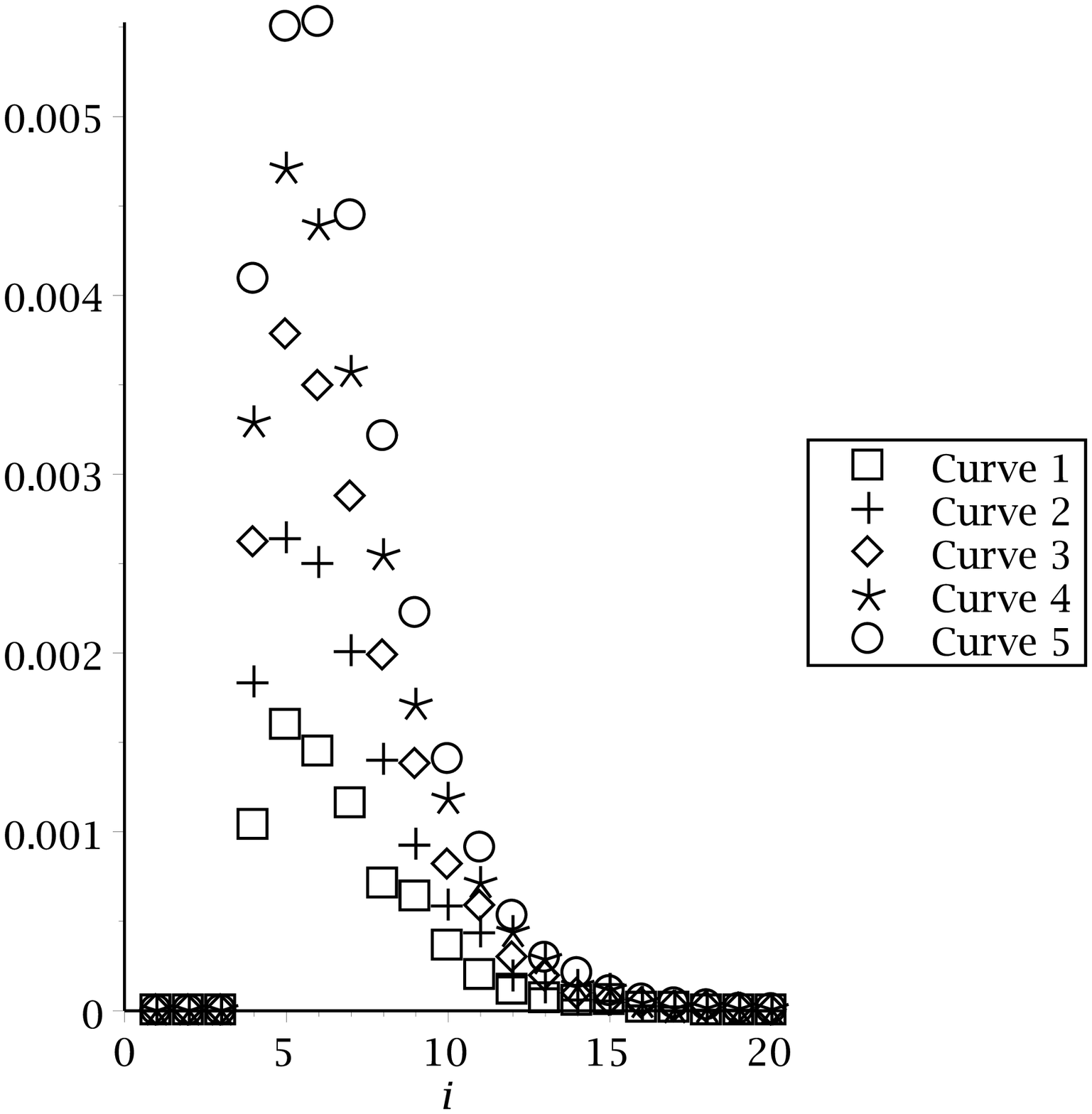}
\end{center}
\caption{ \small  The length-distribution of $\beta$-loops, sample and 
 set-up as in Fig.~\ref{F:loop1}.}
\label{F:loop2}
\end{figure}
%%%
%%%%%%%%%%%%%%%%%%%%%%%%%%%%%%%%%%%%%%%%%%%%%%%%%%%%%%%
%%%

Next we depict the distribution stack-length of uniform 
versus biological RNA-RNA interaction structures obtained from
\cite{Richter}. In both distributions we observe that lower stack 
length appears with high probability, see Fig.~\ref{F:stack1}.  

Finally we present the distribution of $\beta$ stacks versus
that of both, $\alpha$- and $\beta$-stacks, see Fig.~\ref{F:stack2}.
%%%%%%%%%%%%%%%%%%%%%%%%%%%%%%%%%%%%%%%%%%%%%%%%%%%%%%%%%
%%%
\begin{figure}[H]
\begin{center}
\includegraphics[width=0.45\textwidth,height=0.45\textwidth]{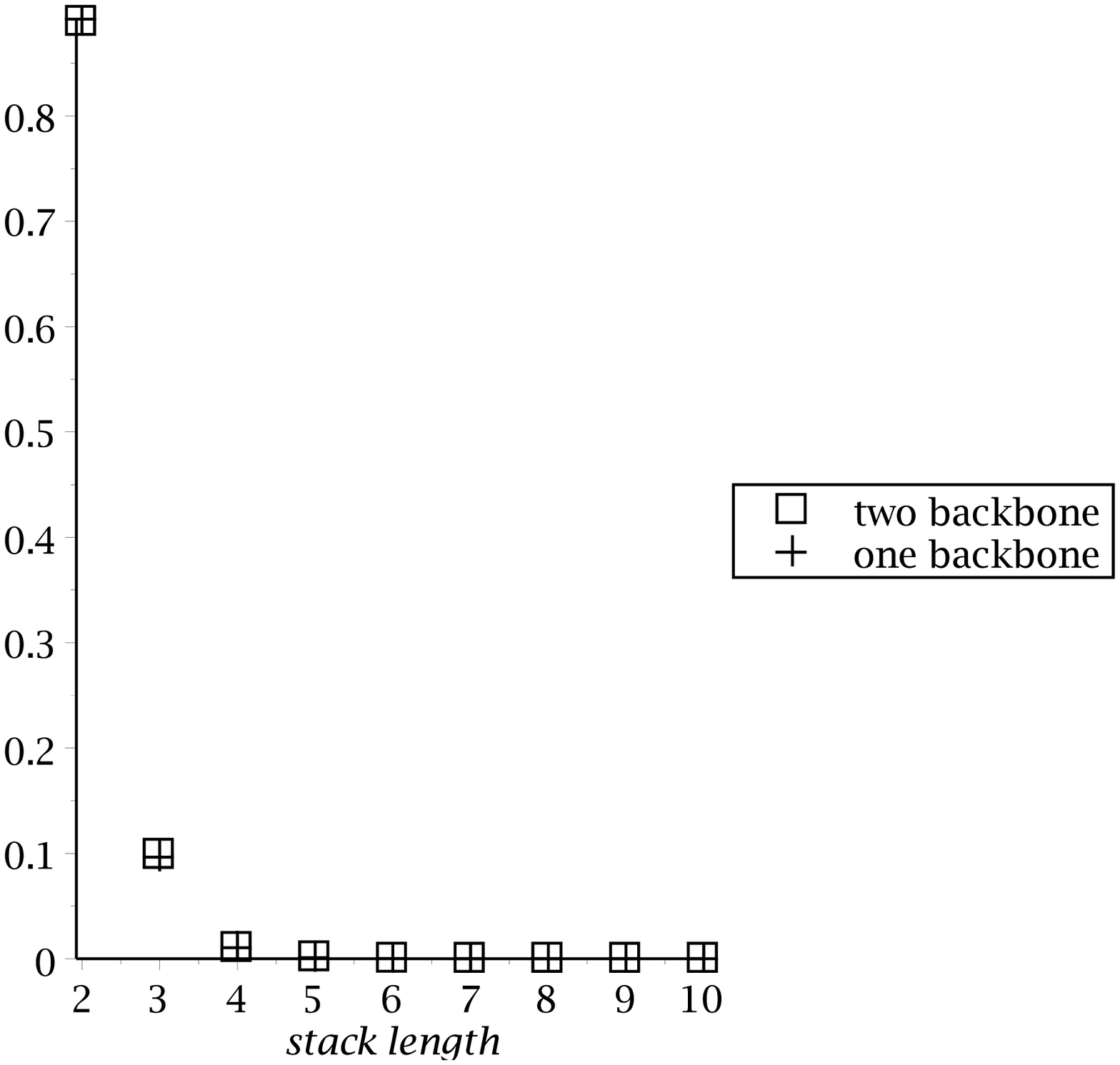}
\includegraphics[width=0.45\textwidth,height=0.45\textwidth]{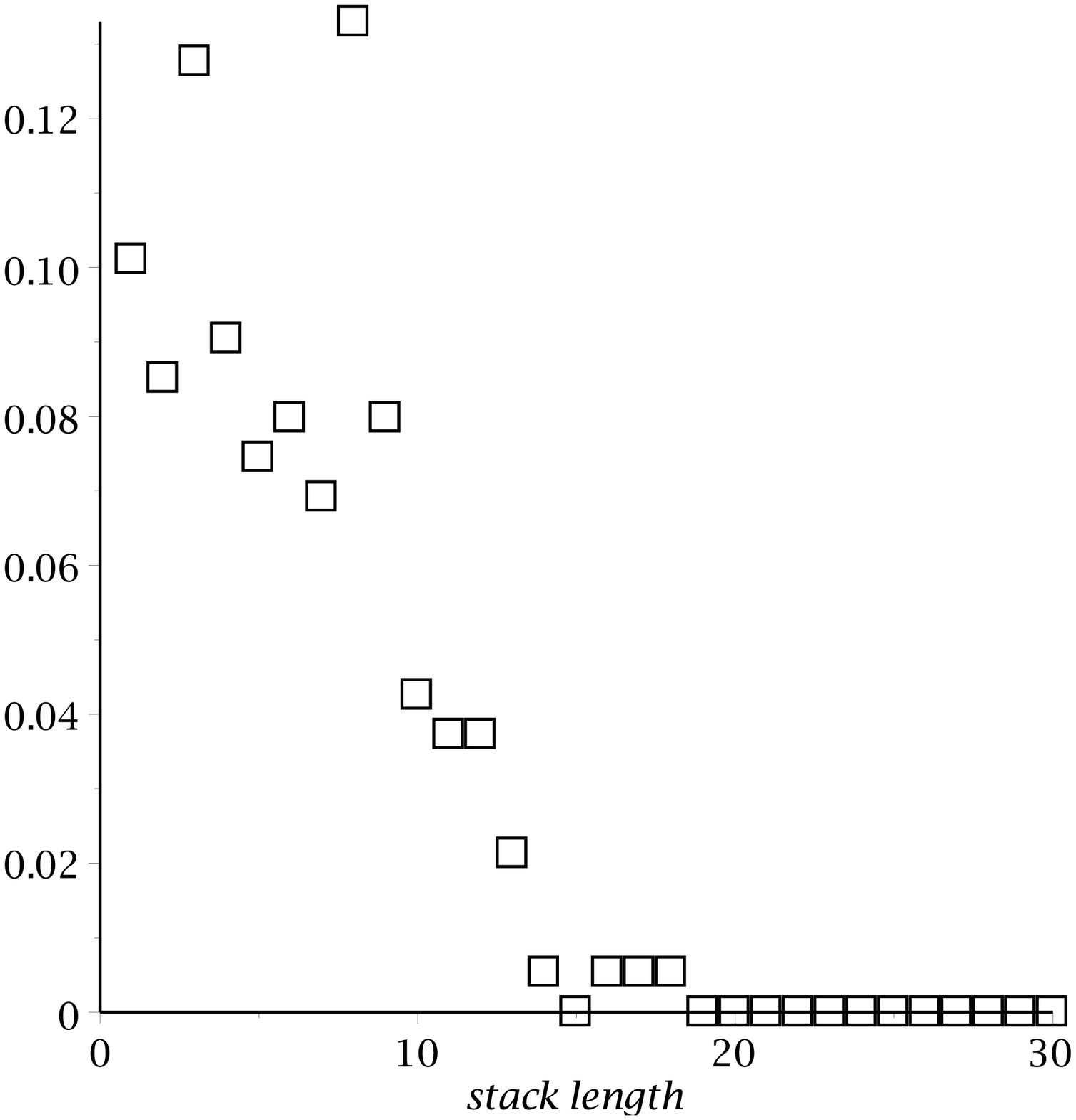}
\end{center}
\caption{ \small left: The distribution of the stack-length 
of $5\times 10^4$ uniformly generated genus $1$ interaction structures 
of length $500$.
right: Distribution of the stack length of biological structures \cite{Richter}.}
\label{F:stack1}
\end{figure}
%%%
%%%%%%%%%%%%%%%%%%%%%%%%%%%%%%%%%%%%%%%
%%%%%%%%%%%%%%%%%%%%%%%%%%%%%%%%%%%%%%%%%%%%%%%%%%%%%%%%%
%%%
\begin{figure}[h]
\begin{center}
\includegraphics[width=0.45\textwidth,height=0.45\textwidth]{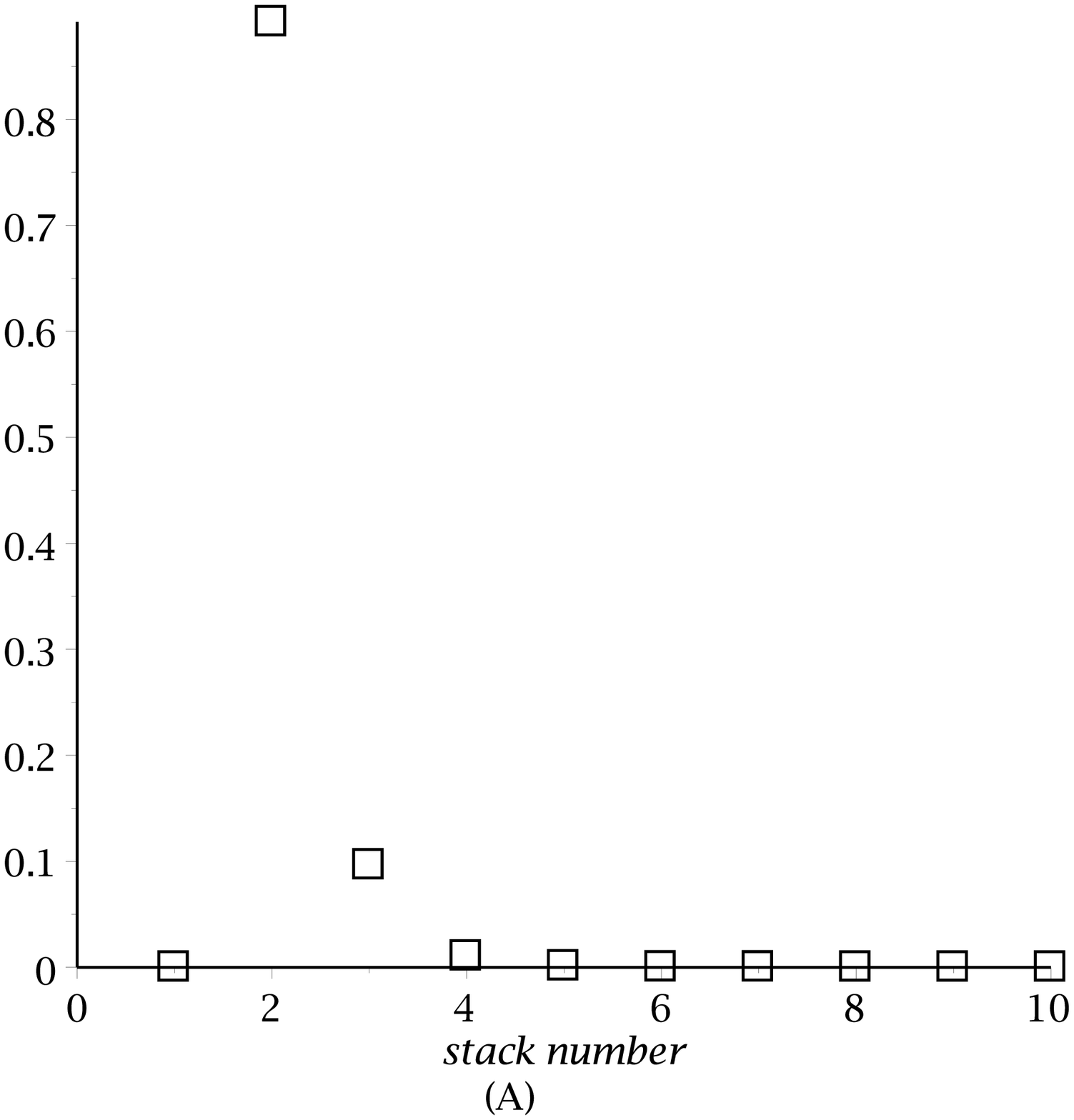}
\includegraphics[width=0.45\textwidth,height=0.45\textwidth]{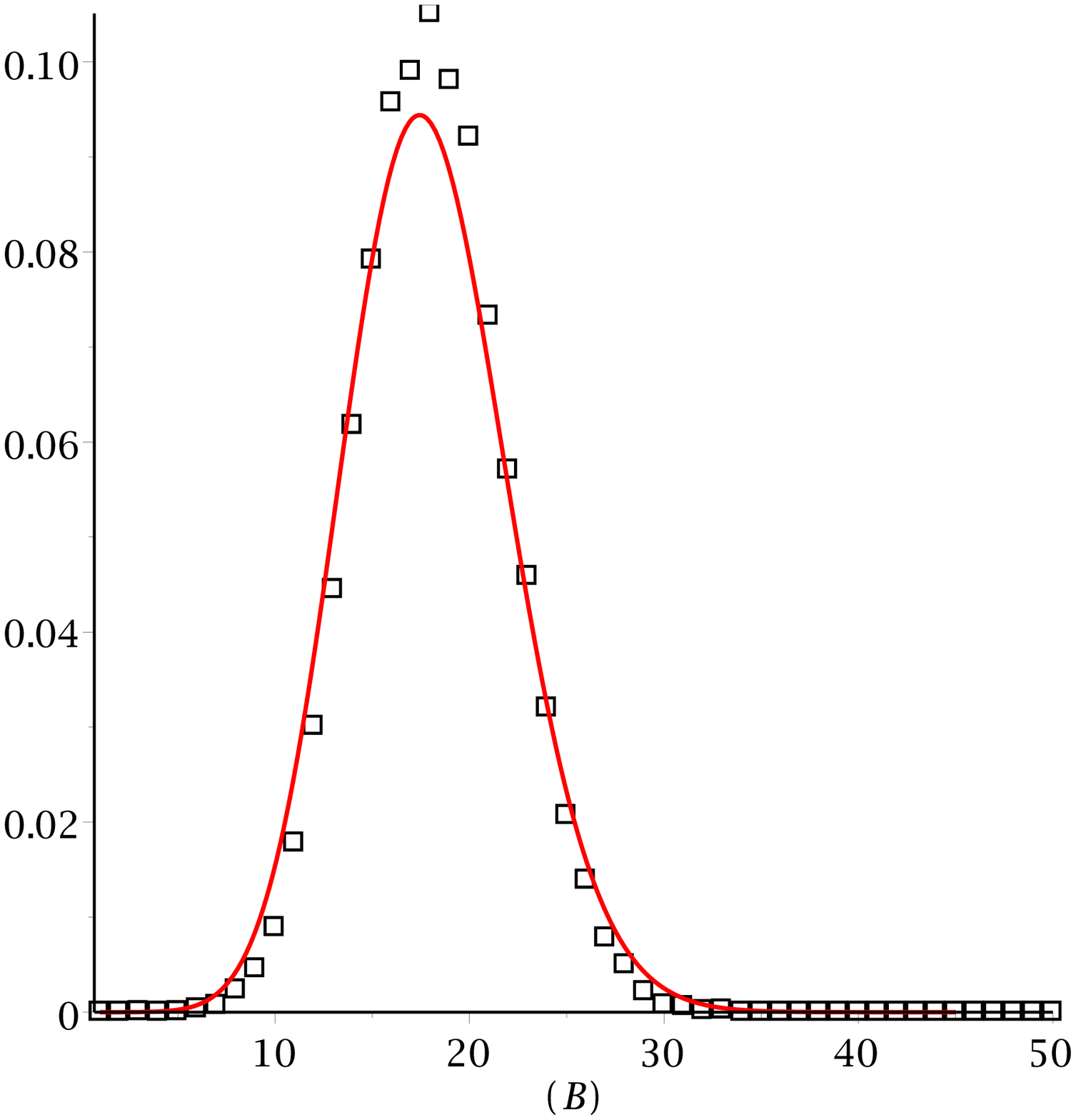}
\end{center}
\caption{ \small $\beta$-stacks versus all stacks: sampling
$10^5$ interaction structures of genus $1$, length $500$, we display:
left:  the distribution of the number of $\beta$ stacks, 
right: the distribution of the number of all stacks.}
\label{F:stack2}
\end{figure}
%%%
%%%%%%%%%%%%%%%%%%%%%%%%%%%%%%%%%%%%%%%

We have in Theorem~\ref{T:main} the blueprint for a novel,
multiple-context free grammar, generating unambiguously RNA-RNA interaction 
structures of genus $g$. This grammar is genuinely topological and can be 
used for a variety of applications. For instance it can be
tailored to produce not uniform but biological interaction structures by means 
of a training set taken from a database of RNA-RNA interaction structures. 
As is standard in stochastic-(multiple) context free grammars, this training set provides 
the probabilities of the rules. It would be then possible to statistically validate
the finding by comparing the derived loop-size statistics from biased 
sampling with that of biological interaction structures.
Another interesting application would arise in the context of functional anotation, 
where via sequencing sites that encode specific pseudoknot RNA like telomerases.
The key objective is the development of local descriptors, as suitable
input for efficient, genome-wide search, which requires deeper, conceptional
understanding of RNA pseudoknots.

%%%
%%%%%%%%%%%%%%%%%%%%%%%%%%%%%%%%%%%%%%%%%%%%%%%%%%%%%%
%%%

{\bf Author contributions}
Hillary S.W. Han obtained an arithmetic proof of Eq.~(\ref{E:2g+2}) based on
\citep{Chapuy:11} and \citep{Reidys-Han} and generated
Figures~(\ref{F:slice1}$(2)$, \ref{F:slice2}, \ref{F:slice3}$(4)$, \ref{F:inflation},
\ref{F:fatgraph}, \ref{F:dual}, \ref{F:case1}, \ref{F:case2},
\ref{F:bb}). Eq.~(\ref{E:induction2}) was 
jointly derived by all authors.
Benjamin M.M.Fu and Christian M.~Reidys derived the bijections, designed
the algorithms, the statistical results and wrote the paper.

{\bf Acknowledgements}
We wish to thank Fenix W.D. Huang and Thomas J.X. Li for discussions.
This work is funded by the Future and Emerging Technologies (FET) programme of the European Commission within the Seventh Framework Programme (FP7), under the
FET-Proactive grant agreement TOPDRIM, FP7-ICT-318121.
\section*{References}

\bibliography{reference}

\end{document}